\documentclass[12pt]{article}

\usepackage{amsmath,amsthm,amsfonts,amssymb, color,xcolor,subcaption,graphicx} 

\usepackage[normalem]{ulem}
\usepackage{todonotes, dsfont}

\newtheorem{theorem}{Theorem}[section]

\newtheorem{proposition}[theorem]{Proposition}
\newtheorem{lemma}[theorem]{Lemma}

\def\cB{\mathcal{B}}
\def\cC{\mathcal{C}}

\def\cF{\mathcal{F}}
\def\cG{\mathcal{G}}
\def\cH{\mathcal{H}}
\def\cI{\mathcal{I}}

\def\cL{\mathcal{L}}

\def\cN{\mathcal{N}}

\def\cZ{\mathcal{Z}}

\def\bE{\mathbb{E}}

\def\bN{\mathbb{N}}
\def\bP{\mathbb{P}}
\def\bR{\mathbb{R}}

\newcommand{\les}{\lesssim}

\newcommand{\fm}{\mathfrak{m}}

\begin{document}

\title{Gaussian fluctuations for the parabolic Anderson model with L\'evy white noise}

\author{Raluca M. Balan\footnote{Corresponding author. University of Ottawa, Department of Mathematics and Statistics, 150 Louis Pasteur Private, Ottawa, Ontario, K1N 6N5, Canada. E-mail address: rbalan@uottawa.ca.} \footnote{Research supported by a grant from the Natural Sciences and Engineering Research Council of Canada.}
\and
Matis Le Gall\footnote{DER de math\'ematiques, ENS Paris-Saclay, Gif-sur-Yvette, France. E-mail address: matis.le\_gall@ens-paris-saclay.fr. }
\and
Jinxin Wang\footnote{University of Ottawa, Department of Mathematics and Statistics, 150 Louis Pasteur Private, Ottawa, Ontario, K1N 6N5, Canada. E-mail address: jwang023@uottawa.ca.}
}

\date{July 2, 2026}
\maketitle

\begin{abstract}
\noindent
In this article, we consider the parabolic Anderson model driven by a L\'evy white noise with finite variance in dimension 1, and we study the asymptotic behaviour of the spatial average of the solution. The main result shows that, with appropriate normalization and centering, the spatial integral converges in distribution to the standard normal distribution, and gives an estimate for the rate of this convergence in the Wasserstein distance, the Kolmogorov distance, and the Fortet-Mourier distance. We also prove the functional limit theorem corresponding to this result.
\end{abstract}


\vspace{1mm}




\section{Introduction}

In the recent years, there has been an increased interest in the literature in examining the asymptotic behaviour as $R\to \infty$ of the spatial average:
\begin{equation}
\label{def-FR}
F_R(t)=\int_{|x|<R} \big(u(t,x)-1\big)dx
\end{equation}
associated with the solution $u$ of a stochastic partial differential equation (SPDE):
\begin{equation}
\label{spde}
\cL u(t,x)=\sigma\big(u(t,x)\big) \dot{W}(t,x), \quad t>0,x\in \bR^d,
\end{equation}
with initial condition 1, where $W$ is a Gaussian noise, $\cL$ is a second-order partial differential operator (usually the heat or the wave operator), and $\sigma$ is a Lipschitz function. The main result in this area, called the {\em Quantitative Central Limit Theorem}, estimates the total variation distance $d_{TV}$ between $F_R(t)/\sigma_R(t)$ and the standard normal distribution, where $\sigma_R^2(t)={\rm Var}(F_R(t))$.   

\medskip

We recall that a (mild) solution of \eqref{spde} is a predictable process $u=\{u(t,x);t\geq 0,x\in \bR^d\}$ which satisfies the integral equation:
\begin{equation}
\label{mild-solution}
u(t,x)=1+\int_0^t \int_{\bR^d} G_{t-s}(x-y) \sigma(u(s,y)) W(ds,dy),
\end{equation}
where $G$ is the fundamental solution of $\cL u=0$.

This line of research has been initiated in the landmark article \cite{HNV20} for the stochastic heat equation (SHE) driven by a space-time Gaussian white noise in dimension 1, and has been subsequently extended to other models. We refer the reader to \cite{BNZ,CKNP22-JFA,CKNP22-AIHP,DNZ,Ebina22,HNVZ,NSZ,NXZ,NZ20,NZ22} for a sample of relevant references.

\medskip

At the core of these investigations lies the following connection between Malliavin calculus and Stein's method for normal approximation: if $F$ is a random variable which can be expressed as a Skorohod integral of a process $v$, with respect to an isonormal Gaussian process $\{W(h)\}_{h \in \cH}$, then
\[
d_{TV}\left(\frac{F}{\sigma},Z\right) \leq \sqrt{{\rm Var} \langle DF,v \rangle_{\cH}},
\]
where $\sigma^2={\rm Var}(F)$ and $D$ is the Malliavin derivative.
In the context of spatial averages related to SPDE mentioned above (i.e. when $F=F_R(t)$), there are two methods for estimating the variance above: 
\begin{description}
\item[(i)] One can use as integrand a process $v$ which arises naturally from the mild formulation \eqref{mild-solution} of the solution. 
\item[(ii)] Alternatively, one can use the intrinsic integrand $v=-DL^{-1} F$ associated with $F$, where $L^{-1}$ is the pseudo-inverse of the Ornstein-Uhlenbeck generator, and then estimate the above variance using the second-order Poincar\'e inequality, which involves the second order Malliavin derivative of $u$. 
\end{description}
Method {\bf (i)} was used in all references mentioned above with Gaussian noise which is colored in space but white in time, while method {\bf (ii)} was used in \cite{BNQZ} for Gaussian noise which is colored in time, and in \cite{BS26} for (SHE) with drift. Both methods rely on the following key estimate for the $p$-th moment of Malliavin derivative of the solution: for any $0<r<t\leq T$, $x,z \in \bR^d$, and $p\geq 2$,
\begin{equation}
\label{key-G}
\|D_{r,z} u(t,x)\|_p \leq C_T G_{t-r} (x-z).
\end{equation}

When $\sigma(u)=u$, $G_{t-r}(x-z)$ is the mean of $D_{r,z}u(t,x)$, which is the first term which appears in its chaos expansion. To show that this term dominates the higher order terms, one calculates the second moment of each term, and then uses hypercontractivity to bound the $p$-th moment by the second moment. 
In the case of a general Lipschitz function $\sigma$, there are other methods for proving \eqref{key-G}. If the noise is white in time, one can write the equation satisfied by $Du$, and then use the Burkholder-Davis-Gundy (BDG) inequality to produce an expression which is amenable to iterations. Another method is to use the product representation:
\begin{equation}
\label{D-prod}
D_{r,z}u(t,x)=\sigma\big(u(r,z)\big) v^{(r,z)}(t,x),
\end{equation}
where the process $\{v^{(r,z)}(t,x);t\geq r,x\in \bR^d\}$ satisfies the integral equation:
\[
v^{(r,z)}(t,x)=G_{t-r}(x-z)+\int_r^t \int_{\bR^d} G_{t-s}(x-y) \sigma'\big( u(s,y)\big) v^{ (r,z)} (s,y) W(ds,dy),
\]
for any $r\geq 0$ and $z \in \bR^d$, and then show that
\[
\|v^{(r,z)}(t,x)\|_p \leq C_T G_{t-r} (x-z),
\]
for all $0\leq r\leq t \leq T$ and $x,z \in \bR^d$. The function $v^{(r,z)}(t,x)=Z(t,x|r,z)$ has been extensively studied in \cite{alberts}
in the case of heat equation with $\sigma(u)=u$, called the parabolic Anderson model (pAm), driven by space-time Gaussian noise in dimension 1. 

\medskip
In article \cite{BZ24}, the QCLT problem was studied for the first time for an SPDE driven by a non-Gaussian process, namely the stochastic wave equation (SWE) with $\sigma(u)=u$ (called the hyperbolic Anderson model), driven by a finite-variance {\em L\'evy white noise} in dimension $d=1$. 
We recall that this noise is given by a collection 
$\{L(\varphi); \varphi\in L^2(\bR_+\times\bR)\}$ of random variables,  defined on a complete probability space $(\Omega,\cF,\bP)$, given by:
\begin{equation}
\label{def-L}
    L(\varphi)=\int_{\bR_+\times\bR\times\bR_0} \varphi(t,x)z \widehat{N}(dt,dx,dz),
\end{equation}
where $N$ is a Poisson random measure on the space
$\bR_+\times\bR\times\bR_0$ with intensity
$dtdx\nu(dz)$, $\widehat{N}$ is the compensated version of $N$, and the measure $\nu$ satisfies: 
\begin{equation}
\label{m2-cond}
m_2:=\int_{\mathbb R_0}|z|^2\nu(dz)<\infty.
\end{equation}
The space $\bR_0=\bR \verb2\2 \{0\}$ is equipped with the distance $d(x,y)=|x^{-1}-y^{-1}|$.

\medskip

The QCLT of \cite{BZ24} was proved using Malliavin calculus on the Poisson space, and a version of the second-order Poincar\'e inequality for this framework, due to \cite{trauthwein25}. At the core of this method lies an estimate for the Malliavin derivative, which is similar to \eqref{key-G}:
\begin{equation}
\label{key-Poisson}
\|D_{r,y,z}u(t,x)\|_p \leq C_T |z|G_{t-r}(x-y),
\end{equation}
for any $0\leq r<t\leq T$, $x,y \in \bR$ and $z\in \bR_0$, where $G_t(x)=\frac{1}{2}1_{\{|x|<t\}}$ is the fundamental solution of the wave equation in dimension 1, and $p\geq 2$ is such that 
\[
m_p:=\int_{\bR_0}|z|^p  \nu(dz)<\infty.
\]
The proof of \eqref{key-Poisson}  uses a product decomposition similar to \eqref{D-prod}, and relies on an application of Rosenthal's inequality, for which it is essential that 
$G_t^p(x)=c_p G_t(x)$. 
The QCLT of \cite{BZ24} has been recently extended in \cite{BZ26} to a general Lipschitz function $\sigma$. 
But the methods of \cite{BZ24,BZ26} do not work for the heat equation.

\medskip

The goal of this article is to fill this gap, and prove the QCLT for (pAm) with finite variance L\'evy white noise. More precisely, we consider:
\begin{equation}
\label{pAm} 
\begin{cases}
\displaystyle 
\frac{\partial u}{\partial t}(t,x)=\frac{1}{2}\frac{\partial^2 u}{\partial x^2}(t,x)+u(t,x)\dot{L}(t,x), \
 t>0,\ x\in\bR, \\
u(0,x)=1, \ x\in\bR.
\end{cases}
\end{equation}
where $L$ is the L\'evy white noise given by \eqref{def-L} and the measure $\nu$ satisfies \eqref{m2-cond}.
The solution of \eqref{pAm} satisfies the integral equation \eqref{mild-solution}, with $\sigma(u)=u$ and $W$ replaced by $L$, in which $G_t$ is the heat kernel:
\[
G_t(x)=\frac{1}{\sqrt{2\pi t}} \exp\left(-\frac{x^2}{2t}\right), \quad t>0,x \in \bR.
\]

In \cite{BN17}, it was proved that equation \eqref{pAm} has a unique solution which satisfies:
\begin{equation}
\label{mom-u}
\sup_{(t,x) \in [0,T] \times \bR} \bE|u(t,x)|^2<\infty.
\end{equation}

We will first prove that $u(t,x)$ has finite $p$-th moments, for any $p \in [2,3)$ such that $m_p<\infty$.
The restriction $p<3$ cannot be lifted, since it comes from the requirement:
\[
\int_0^t \int_{\bR}G_{t-s}^p(x-y)dyds<\infty,
\]
which is the necessary and sufficient condition for $\bE|I_1(t,x)|^p<\infty$, where $I_1(t,x)=\int_0^t \int_{\bR}G_{t-s}(x-y)L(ds,dy)$ is the projection of $u(t,x)$ on the first Poisson chaos space.

\medskip

More importantly, we will show that for any $p\in [2,3)$ such that $m_p<\infty$,
\begin{equation}
\label{key2}
\|D_{r,y,z}u(t,x)\|_p \leq C_T |z| \big(G_{t-r}(x-y)+G_{t-r}^{2/p}(x-y) \big),
\end{equation}
for any $0\leq r<t \leq T$, $x,y\in \bR$ and $z \in \bR_0$. This estimate will play a key role for our developments. Its proof is based on iteration. In the Poisson case, instead of BDG inequality, we use Rosenthal's inequality. This produces two terms: the first term is the same as in the BDG inequality, while the second one involves moments of order $p$. In the $n$-th iterative step of the proof of \eqref{key2}, a careful analysis of the resulting product of $n$ factors (each being a sum of two terms), shows that by a small miracle, all singularities eventually disappear to produce the desired upper bound.

\medskip

We are now ready to state the main results of this article. 
We denote by $\|\cdot\|_p$ the norm in $L^p(\Omega)$ for any $p\geq 1$.
We let $F_R(t)$ be the spatial average given by \eqref{def-FR} of the solution $u$ of \eqref{pAm}, and $\sigma_R^2(t)={\rm Var}(F_R(t))$.

\begin{theorem}
\label{ergodic-cov}
a) {\rm (Ergodicity)} For any $t>0$, $\{u(t,x)\}_{x\in \bR}$ is strictly stationary and ergodic. 
Consequently, by the mean ergodic theorem,
\[
\lim_{R \to \infty}\frac{1}{R}F_R(t) =0 \quad \mbox{a.s. and in $L^2(\Omega)$}.
\]
b) {\rm (Limiting covariance)} For any $t,s>0$,
\begin{equation}
\label{limit-cov}
\lim_{R \to \infty}\frac{1}{R}\bE[F_R(t) F_R(s)]=\Sigma_{t,s}:=4 m_2\int_0^{t \wedge s}e^{\frac{m_2^2 r}{4}}\Phi\left(m_2 \sqrt{\frac{r}{2}} \right)dr,
\end{equation}
where $\Phi(x)=\frac{1}{\sqrt{2\pi}} \int_{-\infty}^x e^{-y^2/2}dy$.
In particular, $\sigma_R^2(t) \sim \Sigma_{t,t} R$ as $R \to \infty$, for any $t>0$.
\end{theorem}

For the next result, we recall that the {\em $1$-Wasserstein distance}, the {\em Fortet-Mourier distance}, and the {\em Kolmogorov distance}  between random variables $X$ and $Y$ defined on the same probability space, are given, respectively, by:
\begin{align*}
d_{W}(X,Y)  &=\sup_{{\rm Lip}(h) \leq 1} \big|\bE[h(X)]-\bE[h(Y)]\big|, \\
d_{FM}(X,Y) &=\sup_{\|h\|_{\infty}+{\rm Lip}(h) \leq 1} \big|\bE[h(X)]-\bE[h(Y)]\big|, \\
d_{K}(X,Y)  &=\sup_{x \in \bR} \big|\bP(X \leq x)-\bP(Y \leq x).\big|.
\end{align*}
where ${\rm Lip}(h):=\sup_{x \not=y}\frac{|h(x)-h(y)|}{|x-y|}$ is the Lipschitz constant of the function $h:\bR \to \bR$.

\begin{theorem}[Quantitative Central Limit Theorem]
\label{QCLT}
Suppose that there exists $p \in (1,\frac{3}{2})$ such that 
\begin{equation}
\label{mp-m2p}
m_{p}<\infty \quad \mbox{and} \quad m_{2p}<\infty.
\end{equation}
Then, for any $t>0$,
\[
{\rm dist}\left( \frac{F_R(t)}{\sigma_R(t)},Z\right) \leq  C_t R^{-(1-\frac{1}{p})}
\]
Here $C_t>0$ is a constant depending on $t$, $Z$ is a standard normal random variable and ${\rm dist}$ is the 1-Wasserstein distance, the Fortet-Mourier distance, or the Kolmogorov distance.
\end{theorem}

The proof of QCLT is based on the second-order Poincar\'e inequality of \cite{trauthwein25} combined with the key estimate \eqref{key2}, and a similar estimate for the second-order Malliavin derivative $D^2 u(t,x)$.  In the case of (pAm) with space-time Gaussian white noise, the rate in QCLT is $R^{-1/2}$.
In the L\'evy case, this rate can be taken to be $R^{-\varepsilon}$ with $\varepsilon \in (0,1/3)$, if the L\'evy measure $\nu$ satisfies $m_p<\infty$ for all $p\in (1,3)$. An example for which $m_p<\infty$ for all $p>0$ is the {\em Gamma white noise} (see Example 1.2 of \cite{BN16}).

\medskip

For the next result, we denote by $D[0,T]$ the space of c\`adl\`ag functions on $[0,T]$, i.e. right-continuous functions $f:[0,T] \to \bR$ which have left limits.

\begin{theorem}[Functional Central Limit Theorem]
\label{FCLT}
Under the hypotheses of Theorem~\ref{QCLT}, for every $R>0$, the process
$\{F_R(t)\}_{t\geq0}$ has a c\`adl\`ag modification (denoted also $F_R$).
Moreover, for any $T>0$,
\[
\frac{1}{\sqrt R}F_R(\cdot)
\xrightarrow{d}
\cG(\cdot) \quad \mbox{in $(D[0,T],J_1)$ as $R \to \infty$},
\]
where $\{\cG(t)\}_{t\geq 0}$ is a zero-mean (continuous) Gaussian process with covariance 
\[
\bE[\cG(t) \cG(s)]=\Sigma_{t,s} \quad \mbox{(given by \eqref{limit-cov})},
\]
Moreover, the convergence also holds in $(D[0,T],U)$, where $U$ denotes the topology of uniform convergence.
\end{theorem}

Theorem \ref{FCLT} introduces an element of novelty in this area, being the first result of this kind for which the limit is in the Skorohod space $D[0,T]$, instead of the usual space $C[0,T]$ of continuous functions. This is due to the fact that in the case of the heat kernel, the function $\varphi_{t,R}(y)=\int_{-R}^R G_{t}(x-y)dx$ converges in $L^2(\bR)$ as $t \to 0$ to $1_{[-R,R]}$, instead of 0 as in the case of the wave kernel. This forces the decomposition $F_R=C_R+M_R$, in which $C_R$ is continuous and $M_R$ is a pure-jump process; see relation \eqref{FR-decomp} below. Therefore, it is not possible to show that $F_R$ has a continuous modification.

\medskip 

Given that (pAm) driven by Gaussian noise is closely related to the Kardar-Parisi-Zhang (KPZ) equation via the Hopf-Cole transformation, the results presented in this article may generate significant interest in the scientific community, being the first to tackle asymptotics for spatial averages related to (pAm) in an impulsive environment which contains jumps. However, to apply the Hopf-Cole transformation, one has to show first that the solution is strictly positive, which is an open problem in the L\'evy setting. (To be the our knowledge, there is no comparison principle for the stochastic heat equation with L\'evy white noise.)

\medskip

This article is organized as follows. In Section \ref{section-prelim}, we present some preliminaries regarding Malliavin calculus on the Poisson space, some moment estimates, and the finite $p$-th moment property of the solution. In Section \ref{section-key}, we present the proof of the key estimate \eqref{key2} for the Malliavin derivative $Du$, and the similar estimate for $D^2 u$. In Section \ref{section-proofs}, we include the proofs of Theorems \ref{ergodic-cov}, \ref{QCLT} and \ref{FCLT}.

\section{Preliminaries}
\label{section-prelim}

In this section, we include some preliminary results.

\subsection{Malliavin calculus}

In this section, we include some basic material about Malliavin calculus with respect to the compensated Poisson random measure $\widehat{N}$.

We consider the Hilbert space $\cH=L^2({\bf Z},\cZ,\fm).$, where
\[
({\bf Z},\cZ,\fm)=\big(\bR_{+} \times \bR \times \bR_0,\ \cB(\bR_+) \otimes \cB(\bR) \otimes \cB(\bR_0), \ {\rm Leb} \times {\rm Leb} \times \nu\big).
\]

$\bullet$ {\bf Chaos Expansion.}
Any random variable $F \in L^2(\Omega)$ which is $\cF^N$-measurable has the {\em Poisson-chaos expansion}:
\begin{equation}
\label{Poisson-chaos}
F=\bE(F)+\sum_{n\geq 1}I_n(f_n), \quad \mbox{for some $f_n \in \cH^{\odot n}$}.
\end{equation}
This series is orthogonal in $L^2(\Omega)$. Here $I_n$ is the multiple integral with respect to $\widehat{N}$ and $\cH^{\odot n}$ is the set of symmetric functions in $\cH^{\otimes n}$.
For any $f \in \cH^{\otimes n}$,
\[
\bE[I_n(f)]=0 \quad \mbox{and} \quad \bE|I_n(f)|^2=n! \|\widetilde{f}\|_{\cH^{\otimes n}}^2,
\]
where $\widetilde{f}$ is the symmetrization of $f$. Moreover, $I_n(f)=I_n(\widetilde{f})$ for any $f \in \cH^{\otimes n}$.

\medskip 
$\bullet$ {\bf Malliavin Derivative.}
For any random variable $F \in L^2(\Omega)$ with chaos expansion \eqref{Poisson-chaos}, we define the {\em Malliavin derivative} of $F$ by:
\[
D_{\xi}F=\sum_{n\geq 1}nI_{n-1}\big(f_n(\cdot,\xi)\big), \quad \mbox{for all} \quad \xi \in {\bf Z},
\]
provided that 
\[
\bE\|DF\|_{\cH}^2=\sum_{n\geq 1}nn! \|\widetilde{f}_n\|_{\cH^{\otimes n}}^2<\infty.
\] In this case, we write $F \in {\rm dom}(D)$.

Similarly, for any integer $k \geq 1$, we define the {\em $k$-th Malliavin derivative} of $F$ by:
\[
D_{\xi_1,\ldots,\xi_k}^kF=\sum_{n\geq k}n(n-1)\ldots (n-k+1)I_{n-k}\big(f_n(\cdot,\xi_1,\ldots,\xi_k)\big), \quad \mbox{for all} \quad \xi_1,\ldots,\xi_k \in {\bf Z},
\]
provided that 
\[
\bE\|D^kF\|_{\cH^{\otimes k}}^2=\sum_{n\geq k}n(n-1)\ldots (n-k+1) n! \|\widetilde{f}_n\|_{\cH^{\otimes n}}^2<\infty.
\] 
In this case, we write $F \in {\rm dom}(D^k)$.

We will use the following result. For the case $k=1$, see Lemma 1.2.3 of \cite{nualart06} for the Gaussian case. 
and Lemma 2.11.(ii) of \cite{BZ26} for the Poisson case.

\begin{lemma}
\label{lem123}
Let $k\geq 1$ be an integer.
Let $(F_n)_{n\geq 1} \subseteq {\rm dom}(D^k)$ be such that $F_n \to F$ in $L^2(\Omega)$, and
\begin{equation}
\label{cond-123}
\sup_{n\geq 1}\bE\|D^kF_n\|_{\cH^{\otimes k}}^2 <\infty.
\end{equation}
Then $F \in {\rm dom}(D^k)$ and $(D^kF_n)_{n\geq 1}$ converges to $D^kF$ in the weak topology of $L^2(\Omega;\cH^{\otimes k})$.
\end{lemma}

\begin{proof}
Note that $D^k$ is a closed operator. Hence, the graph of $D^k$ is closed in $L^2(\Omega) \times L^2(\Omega;\cH^{\otimes k})$, equipped with the product topology, for which $L^2(\Omega;\cH^{\otimes k})$ is equipped with the norm topology. This implies that the graph of $D^k$ is also closed in $L^2(\Omega) \times L^2(\Omega;\cH^{\otimes k})$, equipped with the product topology, for which $L^2(\Omega;\cH^{\otimes k})$ is equipped with the topology induced by the weak convergence. 

By hypothesis \eqref{cond-123}, $(D^kF_n)_{n\geq 1}$ is a bounded sequence in $L^2(\Omega;\cH^{\otimes k})$. By a corollary of Banach-Alaoglu theorem, there exists a subsequence $N'\subset \bN$ such that $D^k F_n$ converges weakly to some $\eta \in L^2(\Omega;\cH^{\otimes k})$, as $n \to \infty,n\in N'$. Because the graph of $D^k$ is closed, we conclude that $F \in {\rm dom}(D^k)$ and $D^k F=\eta$. This proves that $D^k F_n$ converges weakly to some $D^k F$, along the subsequence $N'$. To prove the convergence of the entire sequence, we use the subsequence convergence criterion: in any topological space, a sequence $(x_n)_{n\geq 1}$ converges to $x$ if and only if every subsequence contains a further subsequence which converges to $x$.

\end{proof}

\medskip 

$\bullet$ {\bf Skorohod integral.}
Let $\delta:{\rm dom}(\delta) \to L^2(\Omega)$ be the adjoint of $D$, where ${\rm dom}(\delta)$ is the set of $V \in L^2(\Omega;\cH)$ for which there exists a constant $C=C_V>0$ depending on $V$, such that 
\[
\big|\bE \langle DF,V \rangle_{\cH}\big| \leq C \|F\|_2 \quad \mbox{for any $F \in {\rm dom}(D)$}.
\] 
We say that $\delta(V)$ is 
{\em the Skorohod integral} of $V$ with respect to $\widehat{N}$, and we write
\[
\delta(V)=\int_{\bR_+}\int_{\bR}\int_{\bR_0}
V(t,x,z)\,\widehat N(\delta t,\delta x,\delta z).
\]
By duality, for any $V \in {\rm dom}(\delta)$,
\[
\bE\langle DF,V \rangle_{\cH}=\bE[F \delta(V)] \quad \mbox{for any $F \in {\rm dom}(D)$}.
\]

We recall the {\em Heisenberg commutation principle}. 

\begin{lemma}[Lemma 2.13 of \cite{BZ26}]
\label{Heisenberg}
Assume that $V \in {\rm dom}(\delta)$, $\delta(V) \in {\rm Dom}(D)$, and $V(\xi_0) \in {\rm dom}(D)$ for $\fm$-almost all $\xi_0 \in {\bf Z}$ such that
$\bE \int_{{\bf Z}^2}|D_{\xi_0}V(\xi)|^2\fm(d\xi)\fm(d\xi_0)<\infty$.
Then, for $\fm$-almost all $\xi_0 \in {\bf Z}$, $D_{\xi_0}V \in {\rm dom}(\delta)$ and
\[
D_{\xi_0}\big(\delta (V)\big)=V(\xi_0)+\delta(D_{\xi_0}V). 
\]
\end{lemma}

$\bullet$ {\bf It\^o integral}. Predictable processes on $\Omega \times \bR_{+} \times \bR \times \bR_0$ are defined similarly to those on $\Omega \times \bR_{+} \times \bR$, starting with elementary processes. For any predictable process $V \in L^2(\Omega;\cH)$ we can define the It\^o integral $I(V)$ of $V$ with respect to $\widehat{N}$, and this integral is an isometry: $\bE|I(V)|^2=\bE\|V\|_{\cH}^2$.
The following result states that the Skorohod integral is an extension of the It\^o integral.

\begin{lemma}[Lemma 2.5.(iv) of \cite{BZ24}]
\label{Ito-Skor}
Suppose that $V \in L^2(\Omega;\cH)$ is predictable. Then $V \in {\rm dom}(\delta)$ and $\delta(V)=I(V)$.
\end{lemma}

\subsection{Existence of solution}
The solution of equation \eqref{pAm} is unique and has the Poisson chaos expansion:
\begin{equation}
\label{series}
u(t,x)=1+\sum_{n\geq 1}I_n(F_n(\cdot,t,x)),
\end{equation}
where
\begin{align*}
F_n(t_1,x_1,z_1,\ldots,t_n,x_n,z_n,t,x) &=\prod_{i=1}^n z_i f_n(t_1,x_1,\ldots,t_n,x_n,t,x)\\
f_n(t_1,x_1,\ldots,t_n,x_n,t,x)&=\prod_{i=1}^n G_{t_{i+1}-t_i}(x_{i+1}-x_i)1_{\{0<t_1<\ldots<t_n<t\}} ,
\end{align*}
where we used the convention $t_{n+1}=t$ and $x_{n+1}=x$.

The series \eqref{series} converges in $L^2(\Omega)$ since 
\begin{equation}
\label{series2}
\sum_{n \geq 1}\bE|I_n(F_n(\cdot,t,x))|^2 =\sum_{n \geq 1} n!\|\widetilde{F}_n(\cdot,t,x)\|_{\cH^{\otimes n}}^2=\sum_{n \geq 1}n! m_2^n \|\widetilde{f}_n(\cdot,t,x)\|_{L^2((\bR_+\times \bR)^n)}^2<\infty.
\end{equation}
Moreover, $\bE|u(t,x)|^2=\bE|U(t,x)|^2$, where $U$ is the solution of the parabolic Anderson model driven by space-time Gaussian white noise $W$:
\begin{equation}
\label{pAm-Gauss}
\frac{\partial U}{\partial t}(t,x)=\frac{1}{2}\frac{\partial^2 U}{\partial x^2}(t,x)+\sqrt{m_2} U(t,x)\dot{W}(t,x), \quad  t>0,\ x\in\bR,
 \end{equation}
with initial condition $U(0,x)=1$.

\subsection{Moment inequalities}
In this section, we include some moment inequalities which are used in the sequel.

We recall Rosenthal's inequality for the stochastic integral with respect to $L$. See e.g. Corollary 2.5 of \cite{BN16} or Theorem 1 in \cite{marinelli-rockner14}. 

\begin{proposition}
\label{ros1}
Let $\Phi=\{\Phi(t,x); t\in [0,T], x\in \bR\}$ be a predictable process such that
$\Phi\in L^2(\Omega\times [0,T]\times \bR)$. If $p\geq 2$ is such that $m_p<\infty$, then
\begin{equation}
    \label{Rosenthal}
\bE\left|\int_0^T\int_{\bR}\Phi(t,x)L(dt,dx)\right|^p \leq \cC_p^p \bigg\{  \bE\left(\int_0^T\int_{\bR}|\Phi(t,x)|^2 dxdt\right)^{\frac{p}{2}}
+\bE\int_0^T\int_{\bR}|\Phi(t,x)|^p dxdt \bigg\},
\end{equation}
where $\cC_p^p=2^{p-1}B_p^p\left(m_2^{p/2}\vee m_p\right)$.
\end{proposition}

Here, $B_p>0$ is the constant from Rosenthal's inequality for a c\`adl\`ag square-integrable martingale $\{M(t)\}_{t \geq 0}$ with $M(0)=0$ (see e.g. Lemma 2.1 of \cite{DV90}): for any $p\geq 2$,
$$\|\sup_{s \leq t}|M(s)|\|_p \leq B_p \Big(\| \langle M \rangle(t)^{1/2}\|_p +\|\sup_{s \leq t}|(\Delta M)(s)| \|_p\Big),$$
where $\langle M \rangle$ is the predictable quadratic variation, and $\Delta M$ is the jump process.

Letting $c_p=(2\pi)^{\frac{1-p}{2}} p^{-1/2}$ for any $p>0$, we see that:
\begin{equation}
\label{Gtp}
G_t^p(x)=c_p t^{\frac{1-p}{2}}G_{t/p}(x) \quad \mbox{and} \quad \int_{\bR}G_t^p(x)dx=c_p t^{\frac{1-p}{2}}.
\end{equation}
Therefore, for any $p \in (0,3)$,
\[
\nu_{t,p}:=\int_0^t \int_{\bR} G_{t-s}^p(x-y)dyds=\frac{2 c_p}{3-p} t^{\frac{3-p}{2}}.
\]

\begin{lemma}
Let $p\geq 2$ be such that $m_p<\infty$, and $\Phi=\{\Phi(t,x); t\geq 0, x\in\bR\}$ be
predictable process such that for any $t>0$ and $x\in \bR$,
\[
\int_0^t \int_{\bR }G_{t-s}^2(x-y)|\Phi(s,y)|^2 dyds<\infty. 
\]
Then, for any $t>0$ and $x \in \bR$,
\begin{equation}
    \label{estimate}
\bE\left| \int_0^t\int_{\bR}G_{t-s}(x-y)\Phi(s,y)L(ds,dy) \right|^p \le A_{t,p} \int_0^t\int_{\bR}
H_{t-s}^{(p)}(x-y)\bE|\Phi(s,y)|^p dyds,
\end{equation}
where $A_{t,p}=\cC_p^p\max\{\nu_{t,2}^{\frac{p}{2}-1}, 1\}$, $\cC_p$ is the constant from Proposition \ref{ros1}, and 
\begin{equation}
\label{def-H}
H_t^{(p)}(x)=G_t^2(x)+G_t^p(x). 
\end{equation}
\end{lemma}

\begin{proof}
By Proposition \ref{ros1}, 
\begin{align*}
\bE\left|\int_0^t\int_{\bR}G_{t-s}(x-y)\Phi(s,y)L(ds,dy) \right|^p & \le \cC_p^p \left\{
\bE\left(
\int_0^t\int_{\bR}G_{t-s}^2(x-y)|\Phi(s,y)|^2 dyds\right)^{p/2} \right. \\
& \left.
+\int_0^t\int_{\bR}G_{t-s}^p(x-y)\|\Phi(s,y)\|_p^p dyds \right\}=:
\cC_p^p(I_1+I_2)  
\end{align*}

\noindent By Minkowski's inequality, 
\[
I_1^{2/p}=\left\|\int_0^t\int_{\bR}G_{t-s}^2(x-y)|\Phi(s,y)|^2 dyds \right\|_{p/2} \leq \int_0^t\int_{\bR}G_{t-s}^2(x-y)\|\Phi(s,y)\|_{p}^2 dyds.
\]

Using H\"older's inequality for the finite measure $\mu(ds,dy)=G_{t-s}^2(x-y)dsdy$ on $[0,t] \times \bR$ (whose total mass is $\nu_{t,2}$), we obtain:
\begin{align*}
I_1 & \leq \left(\int_0^t\int_{\bR}G_{t-s}^2(x-y)\|\Phi(s,y)\|_{p}^{2}dyds \right)^{\frac{p}{2}} \le \nu_{t,2}^{\frac{p}{2}-1}\int_0^t\int_{\bR}G_{t-s}^2(x-y)\|\Phi(s,y)\|_{p}^{p}dyds.
    \end{align*}
The conclusion follows.
\end{proof}

In what follows, we will apply \eqref{estimate} to three cases: $\Phi=u$, 
$\Phi=D_{r,y,z}u$ and $\Phi=D^2_{(r_1,y_1,z_1),(r_2,y_2,z_2)}u$.  We denote
\[
T_n(t)=\{\pmb{t}=(t_1,\ldots,t_n) \in [0,t]^n; t_1<\ldots<t_n\}
\]

We will use the following result about beta-type integrals, whose proof follows by induction. 

\begin{lemma}
\label{beta-int}
For any $\beta_0,\beta_1, \ldots,\beta_n>-1$,
\[
\int_{T_n(t)} \prod_{i=0}^n (t_{i+1}-t_i)^{\beta_i} dt_1 \ldots dt_n =\frac{\prod_{i=0}^n\Gamma\left(\beta_i+1\right) \cdot t^{|\beta|+n}}{\Gamma\left(|\beta|+n+1\right)}.
\]
where $|\beta|=\sum_{i=0}^n \beta_i$, and we used the conventions $t_{n+1}=t$ and $t_0=0$.
\end{lemma}

\begin{theorem}
Let $p\in [2,3)$ be such that $m_p<\infty$. Let $u$ be the mild solution of
\eqref{pAm}. Then, 
\begin{equation}
\label{u-pth-moment}
    \sup_{(t,x)\in [0,T]\times \bR} \bE|u(t,x)|^p <\infty \quad \mbox{for any $T>0$}.
\end{equation}
\end{theorem}

\begin{proof}
Let $(u_n)_{n\geq 1}$ be the sequence of Picard iterations, defined by: $u_0(t,x)=1$ and
\begin{equation}
\label{Picard}
u_{n+1}(t,x)=1+\int_0^t \int_{\bR} G_{t-s}(x-y) u_n(s,y)L(ds,dy), \quad n\geq 0.
\end{equation}

We denote $J_n(t,x)=u_{n}(t,x)-u_{n-1}(t,x)$ for $n \geq 1$ and $J_0(t,x)=1$. Then 
\begin{equation}
\label{Picard-J}
J_{n+1}(t,x)=\int_0^t \int_{\bR} G_{t-s}(x-y) J_n(s,y)L(ds,dy), \quad \mbox{for all $n\geq 0$}.
\end{equation}

Using \eqref{estimate}, we obtain that for any $n\geq 0$, $t \in  [0,T]$ and $x \in \bR$,
\[
\bE|J_{n+1}(t,x)|^p \leq A_{T,p} \int_0^t\int_{\bR} H_{t-s}^{(p)}(x-y)\bE|J_n(s,y)|^p dyds,
\]
(We used the fact that $A_{t,p}$ is increasing in $t$, since $p<3$.)
By induction, it follows that:
\begin{equation}
\label{bound-Jn}
\bE |J_n(t,x)|^p  \leq A_{T,p}^{n} R_n(t) \quad \mbox{for any $n\geq 1$},
\end{equation}

where 
\begin{equation*}
R_n(t) :=  
\int_{T_n(t)}
\int_{\bR^n}
\prod_{i=1}^n
H_{t_{i+1}-t_i}^{(p)}(x_{i+1}-x_i) d\pmb{x} d\pmb{t}=\int_{T_n(t)}
\prod_{i=1}^n \left( \int_{\bR}H_{t_{i+1}-t_i}^{(p)}(y_i) dy_i \right) d\pmb{t},
\end{equation*}
where we denoted $\pmb{t}=(t_1,\ldots,t_n)$, $\pmb{x}=(x_1,\ldots,x_n)$ and $\pmb{y}=(y_1,\ldots,y_n)$.
It remains to estimate $R_n(t)$.
By \eqref{Gtp},
\begin{equation}
\label{int-H}
 \int_{\bR} H_t^{(p)}(x) dx=c_2 t^{-\frac{1}{2}}+c_p t^{\frac{1-p}{2}}\leq c_2 \Big( t^{-\frac{1}{2}}+t^{\frac{1-p}{2}}  \Big),
\end{equation}
using the fact that $c_p \leq c_2$.
It follows that
\[
R_n(t) \leq {c_2^n} \int_{T_n(t)} \prod_{i=1}^n \left[ (t_{i+1}-t_i)^{-\frac12} + (t_{i+1}-t_i)^{\frac{1-p}{2}} \right] d\pmb{t}.
\]

We use the following identity: for any $a_1^{(1)},a_1^{(2)},\ldots,a_n^{(1)},a_n^{(2)} \in \bR$,
\[
\prod_{i=1}^n\left(a_i^{(1)}+a_i^{(2)}\right) =\sum_{\ell_1,\dots,\ell_n=1}^2
\prod_{i=1}^n a_i^{(\ell_i)}.
\]

We obtain:
\[
R_n(t) \leq
c_2^n
\sum_{\ell_1,\dots,\ell_n=1}^2
\int_{T_n(t)}
\prod_{i=1}^n
(t_{i+1}-t_i)^{\beta_i^{(\pmb{\ell})}}
dt_1\cdots dt_n,
\]
where $\pmb{\ell}=(\ell_1,\ldots,\ell_n)\in\{1,2\}^n$ and
\begin{equation}
\label{def-beta}
\beta_i^{(\pmb{\ell})}
=\left\{
\begin{array}{ll} 
-1/2 & \mbox{if $\ell_i=1$} \\
(1-p)/2 & \mbox{if $\ell_i=2$}
\end{array} \right. \quad \mbox{for any $i=1,\ldots,n$}.
\end{equation}

To evaluate the last integral, we use Lemma \ref{beta-int} and the fact that $p<3$. Hence,
\begin{align*}
R_n(t)
& \leq c_2^n
\sum_{\ell_1,\dots,\ell_n=1}^2
\frac{
\prod_{i=1}^n\Gamma\left(\beta_i^{(\pmb{\ell})}+1\right)
\cdot
t^{\sum_{i=1}^{n}\beta_i^{(\pmb{\ell})}+n}
}{
\Gamma\left(\sum_{i=1}^{n}\beta_i^{(\pmb{\ell})}+n+1\right)
}.
\end{align*}

For any fixed $\pmb{\ell}=(\ell_1,\ldots,\ell_n)\in\{1,2\}^n$,
$\beta_i^{(\pmb{\ell})}+1$ is either $\frac12$ or $\frac{3-p}{2}$. 
Since $0<\frac{3-p}{2} \leq \frac{1}{2}$ and the Gamma function is decreasing on $(0,\frac{1}{2}]$, it follows that 
\[
\prod_{i=1}^n\Gamma\left(\beta_i^{(\pmb{\ell})}+1\right)\leq
\Gamma\left(\frac{3-p}{2}\right)^n.
\]

Since $p\in[2,3)$,we have
\begin{equation}
\label{sum-beta}
\frac{1-p}{2}n
\leq
\sum_{i=1}^n \beta_i^{(\pmb{\ell})}
\leq
-\frac n2
\end{equation}  
Using the fact that the Gamma function is increasing on $(x_0,\infty)$ for some $x_0 \in (1,2)$, it follows that $\Gamma\left(\sum_{i=1}^n \beta_i^{(\pmb{\ell})}+n+1\right) \geq \Gamma \left(\frac{3-p}{2}n+1 \right)$ for all $n > N_0$, for some $N_0\geq 1$ large enough.

We use the elementary inequality:
\begin{equation}
\label{elem}
t^b \leq T^{b-a} t^a \quad \mbox{for any $t \in [0,T]$ and $a \leq b$}.
\end{equation}
Hence, for any $t \in [0,T]$,
\[
t^{\sum_{i=1}^n \beta_i^{(\pmb{\ell})}+n} \leq C_{T,\pmb{\ell},n}t^{\frac{3-p}{2}n},
\]
where $C_{T,\pmb{\ell},n}=T^{\sum_{i=1}^n \beta_i^{(\pmb{\ell})}+n-\frac{3-p}{2}n}$. We bound the constant $C_{T,\pmb{\ell},n}$, using \eqref{sum-beta} and \eqref{elem} . Hence,
\[
t^{\sum_{i=1}^n \beta_i^{(\pmb{\ell})}+n} \leq C_{T,p}^n t^{\frac{3-p}{2}n} \quad \mbox{for any $t \in [0,T]$},
\]
where $C_{T,p}=T^{\frac{p-2}{2}}$ if $T \geq 1$ and $C_{T,p}=1$ if $T \in (0,1)$.

All these bounds are valid for any  fixed $\pmb{\ell}=(\ell_1,\ldots,\ell_n)\in\{1,2\}^n$.
We obtain that:
\[
 R_n(t)\le \left(2c_2C_{T,p}\Gamma\left(\frac{3-p}{2}\right)\right)^n \frac{t^{\frac{3-p}{2}n}}{\Gamma\left(\frac{3-p}{2}n+1\right)} \quad \mbox{for all $n> N_0$}.
\]

We use the fact that for any $a>0$, there exist some constants $C_{a,1},C_{a,2}>0$ depending on $a$ such that $C_{a,1}^n (n!)^a \leq \Gamma(an+1) \leq C_{a,2}^n (n!)^a$ for all $n \geq 1$. Hence,
\begin{equation}
\label{bound-Rn}
    R_n(t)\le C^n\frac{t^{\frac{3-p}{2}n}}{(n!)^{\frac{3-p}{2}}} \quad \mbox{for all $n> N_0$},
\end{equation}
where $C$ is a constant that depends on $(T,p)$. Using \eqref{bound-Jn} and \eqref{bound-Rn}, we obtain that:
\begin{equation}
\label{bound-Jn}
\bE|J_{n}(t,x)|^p  \leq  
(CA_{T,p})^n   \frac{t^{ \frac{3-p}{2} n } }{(n!)^{\frac{3-p}{2}}}, \quad \mbox{for all $n> N_0$}.
\end{equation}

This shows that
\[
\sum_{n\geq 1} \sup_{(t,x)\in [0,T] \times \bR}\|J_n(t,x)\|_p<\infty.
\] 
Hence, $\{u_n(t,x)\}_{n\geq 1}$ is a Cauchy sequence in $L^p(\Omega)$, uniformly in $(t,x) \in [0,T] \times \bR$. Since its limit is the solution $u(t,x)$ of equation \eqref{pAm}, we have:
\[
\sup_{(t,x) \in [0,T] \times \bR} \bE|u_n(t,x)-u(t,x)|^p \to 0 \quad \mbox{as} \quad n\to \infty.
\]

By triangular inequality and \eqref{bound-Jn}, for all $n> N_0$,
\begin{align*}
\|u_n(t,x)\|_p & \leq 1+\sum_{k=1}^n\|J_k(t,x)\|_p \leq M+\sum_{k=N_0+1}^n (CA_{T,p})^{k/p}   \frac{t^{ \frac{3-p}{2p} k } }{(k!)^{\frac{3-p}{2p}}} \\
&  \leq M +\sum_{k\geq 1} (CA_{T,p})^{k/p}   \frac{t^{ \frac{3-p}{2p} k } }{(k!)^{\frac{3-p}{2p}}}=:C_0 <\infty,
\end{align*}
where 
\[
M:=
1+\sum_{k=1}^{N_0}
A_{T,p}^{k/p}
\sup_{t\in[0,T]}R_k(t)^{1/p}
<\infty,
\]
since $R_n$ is continuous. 
The conclusion follows.
\end{proof}

\section{Key estimates on Malliavin derivatives}
\label{section-key}

In this section, we provide some key estimates for the $p$-th moments of the first and second Malliavin derivatives of the solution of equation \eqref{pAm}, which will play an essential role in the proof of the QCLT. These bounds can be extended to Malliavin derivative $D^k u$, for any $k\geq 1$. To simplify the presentation, we include only the cases $k=1$ and $k=2$. 

 Unlike the Gaussian case, these estimates will be given by the function:
\[
g_{t}^{(p)}(x):=G_t(x)+G_t^{2/p}(x), \quad t>0,x\in \bR,
\]
where $p \in [2,3)$ is such that $m_p<\infty$. The restriction $p<3$ is natural and comes from the requirement: 
\[
\int_0^t \int_{\bR} G_{t-s}^p(x-y)dyds<\infty,
\]
which is the necessary and sufficient condition for $\bE|v(t,x)|^p<\infty$, where $v(t,x)=I_1(f_1(\cdot,t,x))$ is the first term in the chaos expansion \eqref{series} of $u(t,x)$ (and also the solution of the stochastic heat equation with additive noise $L$). This restriction can be relaxed by considering an equation with {\em L\'evy colored noise}. We postpone this case for future work.

For both results, we first prove the desired bound for the sequence $(u_n)_{n\geq 0}$ of Picard iterations (given by \eqref{Picard}), then we conclude that the bound holds for the solution itself using a weak convergence argument.

For any integer $n\geq 1$ and for any $0<r<t$, we consider the {\em restricted simplex}:
\[
T_n(r,t)=\{(t_1,\ldots,t_n) \in [r,t]^n; t_1<\cdots<t_n\}.
\]

The next result provides the key estimate for the first Malliavin derivative of the solution.

\begin{theorem}
\label{key-th-Du}
a) For any $t\geq 0$ and $x \in \bR$, $u(t,x)\in {\rm dom}(D)$.

b)  Let $p\in [2,3)$ be such that $m_p<\infty$. For any $0\le r< t\le T$, $x,y\in\bR$ and $z\in\bR_0$,
\begin{equation}
\label{Du-pth-moment}
 \| D_{r,y,z}u(t,x)\|_p\le C_{T,p}|z|g_{t-r}^{(p)}(x-y),
\end{equation}
where $C_{T,p}>0$ is a constant depending on $(T,p)$, which is non-decreasing in $T$.
\end{theorem}

\begin{proof} 
Note that $Du_0=0$. Moreover, for all  $n\geq 0, u_n \in {\rm dom}(D)$, as its chaos expansion is finite by the induction hypothesis.
We will prove below that for any $p\in [2,3)$ such that $m_p<\infty$, and for any $0 \leq r<t\leq T $, $x,y \in \bR$, $z\in \bR_0$ and $n \geq 1$, 
\begin{equation}
\label{Du-G}
\bE\left|D_{r,y,z}u_n(t,x)\right|^p \leq C_{T,p}^p |z|^p H_{t-r}^{(p)}(x-y),
\end{equation}
where $C_{T,p}>0$ is a constant depending on $(T,p)$ which is non-decreasing in $T$. 
In particular, for $p=2$, this implies that
\[
\sup_{n\geq 1}\sup_{(t,x)\in [0,T] \times \bR}\bE \|D u_n(t,x)\|_{\cH}^2 <\infty.
\]

Indeed, by \eqref{Du-G} and \eqref{int-H}, we obtain that for any $n \geq 1$, $t \in [0,T]$ and $x \in \bR$,
\begin{align*}
\bE \|D u_n(t,x)\|_{\cH}^2 & =\int_{0}^t  \int_{\bR} \int_{\bR_0} \bE |D_{r,y,z} u_n(t,x)|^2 drdy\nu(dz) \leq m_2 C_{T,2}^2  \int_0^t \int_{\bR} H_{t-r}^{(2)}(x-y) dydr\\
& \leq m_2 C_{T,2}^2 c_2 \int_0^t 2(t-r)^{-1/2} dr=4m_2 C_{T,2}^2c_2  t^{1/2}.
\end{align*}
We apply lemma \ref{lem123} to $F_n=u_n(t,x)$ and $F=u(t,x)$ for fixed $(t,x)\in \bR_{+}\times \bR$. We conclude that
$u(t,x)\in {\rm dom}(D)$ and $\{Du_n(t,x)\}_{n\geq 1}$ converges to $Du(t,x)$ in the weak topology
$L^2(\Omega;\mathcal{H})$. This proves part a).

Taking power $1/p$ in \eqref{Du-G} and using the inequality $(x+y)^a \leq x^a+y^a$ for any $x,y>0$ and $a\in (0,1)$, we infer that for any $0 \leq r<t\leq T$, $x,y \in \bR$ and $z\in \bR_0$,
\[
\sup_{n\geq 0}\|D_{r,y,z}u_n(t,x)\|_p \leq C_{T,p} |z| g_{t-r}^{(p)}(x-y).
\]
Then, the conclusion in part b) follows by Lemma A.1 of \cite{BS26}.

\medskip

We now prove \eqref{Du-G}. Similarly to the proof of Proposition 3.1 of \cite{BZ26}, it can be proved that for $\fm$-almost $\xi \in {\bf Z}$, $D_{\xi}u_n$ has a predictable modification (denoted also $D_{\xi}u_n$). We work with this modification. Using the  recurrence relation \eqref{Picard}, Heisenberg commutation principle (Lemma \ref{Heisenberg}) and the fact that It\^o integral coincides with Skorohod integral if the integrand is {\em predictable} (Lemma \ref{Ito-Skor}), it follows that
 \begin{equation}
 \label{Du}
D_{r,y,z}u_{n+1}(t,x)=G_{t-r}(x-y) z u_n(r,y)+ \int_r^t\int_{\bR} G_{t-s_1}(x-y_1)D_{r,y,z}u_n(s_1,y_1) L(ds_1,dy_1).
    \end{equation}

Applying \eqref{estimate} with $\Phi=Du_n$, we obtain that for any $0\leq r\leq t$, $x,y \in \bR$ and $z \in \bR_{0}$,
\begin{align*}
& \bE\left|D_{r,y,z}u_{n+1}(t,x)\right|^p \\
& \leq 2^{p-1} \left\{ G_{t-r}^p(x-y)|z|^p \bE |u_n(r,y)|^p+ \bE\left|
\int_r^t\int_{\bR} G_{t-s_1}(x-y_1) D_{r,y,z}u_n(s_1,y_1)
L(ds_1,dy_1) \right|^p \right\} \\
&\le 2^{p-1}\left\{ G_{t-r}^p(x-y)|z|^p K_{T,p}^p+ A_{T,p} \int_r^t\int_{\bR} H_{t-s_1}^{(p)}(x-y_1) \bE|D_{r,y,z}u_n(s_1,y_1)|^p ds_1dy_1
\right\},
\end{align*}    
where $K_{T,p}^p=\sup_{n\geq 0}\sup_{(t,x)\in [0,T] \times \bR}\bE|u_n(t,x)|^p$. By induction, we infer that for any $n\geq 1$,
\begin{equation}
\label{Dun-bound}
\bE\left|D_{r,y,z}u_{n}(t,x)\right|^p \leq 2^{p-1}|z|^p K_{T,p}^p \Big\{ G_{t-r}^p(x-y) + \sum_{k=1}^{n-1}\left(2^{p-1}A_{T,p}\right)^k S_k \Big\},
\end{equation}
where
\begin{equation}
\label{def-Sn}
S_n:=S_n(r,t,x-y) =   \int_{T_n(r,t)} \int_{\bR^n} \prod_{i=1}^{n} H_{t_{i+1}-t_i}^{(p)}(x_{i+1}-x_i) G_{t_1-r}^p(x_1-y) d\pmb{x} d\pmb{t},
\end{equation}
with $t_{n+1}=t$ and $x_{n+1}=x$. Using the change of variables $y_i=x_{i+1}-x_i$ for $i=1,\dots,n$, we obtain 
\begin{equation*}
\begin{aligned}
    S_n &= \int_{T_n(r,t)} \int_{\bR^n} \prod_{i=1}^{n} H_{ t_{i+1}-t_i }^{(p)}(y_i) G_{t_1-r}^p(x-y-y_1-\dots-y_n)d\pmb{y}  d\pmb{t}\\  
    &=\sum_{\ell_1,\dots,\ell_n=1}^2 \int_{T_n(r,t)} \int_{\bR^n} \prod_{i=1}^{n} G_{t_{i+1}-{t_i}}^{(\ell_i)}(y_i)G_{t_1-r}^p(x-y-\sum_{i=1}^n y_i) d\pmb{y}  d\pmb{t},
\end{aligned}
\end{equation*}
where $G_{t}^{(1)}(x) :=G_{t}^{2}(x)$ and $G_{t}^{(2)}(x)  :=G_{t}^{p}(x)$. For any $\pmb{\ell}=(\ell_1,\ldots,\ell_n) \in \{1,2\}^n$ fixed, we consider a sequence $(X_1^{(\pmb{\ell})},\ldots, X_n^{(\pmb{\ell})})$ of independent random variables such that
\[
X_i^{(\pmb{\ell})}\sim \cN\left(0,\frac{t_{i+1}-t_i}{2}\right) \quad \mbox{if $\ell_i=1$} \quad \mbox{and} \quad X_i^{(\pmb{\ell})}\sim \cN\left(0,\frac{t_{i+1}-t_i}{p}\right) \quad \mbox{if $\ell_i=2$},
\]
and we let $f_{X_i^{(\pmb{\ell})}}(x)$ be the density of $X_i^{(\pmb{\ell})}$. We let $\beta_i^{(\pmb{\ell})}$ be given by \eqref{def-beta}.

Using the first relation in \eqref{Gtp} and the fact that $c_p\leq c_2$, it follows that
\begin{align}
\nonumber
 S_n &\le c_2^n \sum_{\ell_1,\dots,\ell_n=1}^2 \int_{T_n(r,t)}  \prod_{i=1}^{n} (t_{i+1}-t_i)^{\beta_i^{(\pmb{\ell})}} \left( \int_{\bR^n} \prod_{i=1}^{n} f_{X_i^{(\pmb{\ell})}}(y_i)G_{t_1-r}^p(x-y-\sum_{i=1}^n y_i) d\pmb{y} \right) d\pmb{t}\\
 \nonumber
  &= c_2^n \sum_{\ell_1,\dots,\ell_n=1}^2 \int_{T_n(r,t)}  \prod_{i=1}^{n} (t_{i+1}-t_i)^{\beta_i^{(\pmb{\ell})}} \bE\left[G_{t_1-r}^p(x-y-\sum_{i=1}^n X_i^{(\pmb{\ell})}) \right] d\pmb{t}\\
 \label{Snt}
 &=: c_2^{n} \sum_{\ell_1,\dots,\ell_n=1}^2  \cI_n(\pmb{\ell}), \quad \mbox{where} \quad \pmb{\ell}=(\ell_1,\ldots,\ell_n). 
\end{align}

We estimate the integral $\cI_n(\pmb{\ell})$.
Note that $\sum_{i=1}^nX_i^{(\pmb{\ell})}\sim \cN(0,\gamma^{(\pmb{\ell})})$, where 
\begin{equation*}
 \gamma^{(\pmb{\ell})}:= \sum_{i=1,l_i=1}^n \frac{t_{i+1}-t_i}{2} +  \sum_{i=1,l_i=2}^n \frac{t_{i+1}-t_i}{p}.
\end{equation*}
An important observation is that 
\begin{equation}
\label{gam-bound}
\frac{t-t_1}{p} \leq \gamma^{(\pmb{\ell})} \leq \frac{t-t_1}{2} \quad \mbox{for any $\pmb{\ell} \in \{1,2\}^n$}.
\end{equation}

Using the semigroup property, it follows that
\begin{align*}
& \bE\left[G_{t_1-r}^p(x-y-\sum_{i=1}^n X_i^{(\pmb{\ell})}) \right]= \int_\bR G_{t_1-r}^p(x-y-z)G_{\gamma^{(\pmb{\ell})}}(z) dz\\
& \quad =c_p (t_1-r)^{\frac{1-p}{2}} \int_\bR G_{ \frac{t_1-r}{p} }(x-y-z)G_{\gamma^{(\pmb{\ell})}}(z) dz=c_p (t_1-r)^{\frac{1-p}{2}} G_{\frac{t_1-r}{p}+\gamma^{(\pmb{\ell})}}(x-y)\\
& \quad \le c_p (t_1-r)^{\frac{1-p}{2}}    \sqrt{\frac{\frac{t-r}{2}}{\frac{t_1-r}{p}+\gamma^{(\pmb{\ell})}}} G_{\frac{t-r}{2}}(x-y) \le c_p (t_1-r)^{\frac{1-p}{2}} \sqrt{\frac{p}{2}} G_{\frac{t-r}{2}}(x-y)\\
& \quad =\frac{c_p}{c_2} \sqrt{\frac{p}{2}}  (t_1-r)^{\frac{1-p}{2}}  (t-r)^{\frac{1}{2}} G_{t-r}^2(x-y).
\end{align*}
where for the first inequality we used the upper bound in \eqref{gam-bound} and the inequality 
\begin{equation}
\label{sqrt(t over s)}   
G_s(x)\le \sqrt{\frac{t}{s}}G_t(x) \quad \mbox{for} \quad 0<s\le t,
\end{equation}
and for the second inequality we used the lower bound in \eqref{gam-bound}. Recall also that $c_p\leq c_2$. Thus, for any $n \geq 1$ and $\pmb{\ell} \in \{1,2\}^n$,
\begin{align}
\nonumber
\cI_n(\pmb{\ell}) &\leq  \sqrt{\frac{p}{2}} (t-r)^{\frac{1}{2}}G_{t-r}^2(x-y)\int_{T_n(r,t)}  \prod_{i=1}^{n} (t_{i+1}-t_i)^{\beta_i^{(\pmb{\ell})}} (t_1-r)^{\frac{1-p}{2}} d\pmb{t}\\
\nonumber
& = \sqrt{\frac{p}{2}} (t-r)^{\frac{1}{2}}G_{t-r}^2(x-y)\int_{T_n(t-r)}  (t-r-s_n)^{\beta_n^{(\pmb{\ell})}}\prod_{i=1}^{n-1} (s_{i+1}-s_i)^{\beta_i^{(\pmb{\ell})}} s_1^{\frac{1-p}{2}} d\pmb{s} \\
\label{bound-Il}
& = \sqrt{\frac{p}{2}} (t-r)^{\frac{1}{2}}G_{t-r}^2(x-y) \frac{\prod_{i=0}^n\Gamma\left(\beta_i^{\pmb{(\ell)}}+1\right) \cdot (t-r)^{\sum_{i=0}^n\beta_i^{\pmb{(\ell)}}+n}}{\Gamma\left(\sum_{i=0}^n\beta_i^{(\pmb{\ell})}+n+1\right)},
\end{align}
where for the last equality we used Lemma \ref{beta-int} with $\beta_0^{(\pmb{\ell})}=\frac{1-p}{2}$.

\medskip

We fix $n\geq 1$ and $\pmb{\ell} \in \{1,2\}^n$.  We consider two cases.

\medskip
{\em Case 1.} There exists $i_0=1,\ldots,n$ such that $\ell_{i_0}=1$. Then $\beta_{i_0}^{(\pmb{\ell})}=-1/2$. 
Since
$\beta_i^{(\pmb{\ell})}+1 \in \{\frac{1}{2}, \frac{3-p}{2}\}$ for $i=1,\ldots,n$, 
$\beta_0^{(\pmb{\ell})}+1=\frac{3-p}{2}$, and
the Gamma function is decreasing
on $(0,\frac{1}{2})$, we have
\begin{equation}
\label{Gam1}
\prod_{i=0}^n\Gamma\left(\beta_i^{(\pmb{\ell})}+1\right)
\leq
\Gamma\left(\frac{3-p}{2}\right)^{n+1}.
\end{equation}

Hence,
\[
\sum_{i=0}^n\beta_i^{(\pmb{\ell})}=\frac{1-p}{2}+\sum_{i=1,i\not=i_0}^n \beta_i^{(\pmb{\ell})}+ \beta_{i_0}^{(\pmb{\ell})}\geq \frac{1-p}{2}n-\frac{1}{2}
\]
and
\begin{equation}
\label{beta-ineq1}
\sum_{i=0}^n\beta_i^{(\pmb{\ell})} +n \geq \frac{3-p}{2}n-\frac{1}{2}.
\end{equation}

Using again the elementary inequality \eqref{elem}, we get, for any $0\leq r<t\leq T$,
\[
(t-r)^{\sum_{i=0}^n\beta_i^{(\pmb{\ell})}+n}
\leq
T^{\sum_{i=0}^n\beta_i^{(\pmb{\ell})}+n-\frac{3-p}{2}n+\frac{1}{2}}
(t-r)^{\frac{3-p}{2}n-\frac{1}{2}}
\]

Since $\beta_i^{(\pmb{\ell})} \leq -\frac{1}{2}$ for all $i=1,\ldots,n$, we have the upper bound:
\[
\sum_{i=0}^n\beta_i^{(\pmb{\ell})}+n\leq \frac{1-p}{2}+\frac{n}{2}
\]
Hence, the exponent of $T$ is bounded above by $\frac{1-p}{2}+\frac{n}{2}-\frac{3-p}{2}n+\frac{1}{2}=\frac{p-2}{2}(n-1)\leq \frac{p-2}{2}n$. It follows that
\begin{equation}
\label{ineq-t}
(t-r)^{\sum_{i=0}^n\beta_i^{(\pmb{\ell})}+n} \leq B_{T,p}^n (t-r)^{\frac{3-p}{2}n-\frac{1}{2}},
\end{equation}
where $B_{T,p}=T^{\frac{p-2}{2}}$ if $T \geq 1$ and $B_{T,p}=1$ if $T \in (0,1)$.

Recall that the Gamma function is increasing on $(x_0,\infty)$ for $x_0 \approx 1.46$. Hence there exists an integer $N_0 \geq 1$ (depending on $p$) such that 
$\frac{3-p}{2}n+\frac{1}{2}>x_0$ for all $n>N_0$. Using \eqref{beta-ineq1}, it follows that for all $n>N_0$,
\begin{equation}
\label{Gam2}
\Gamma\left(\sum_{i=0}^n\beta_i^{(\pmb{\ell})}+n+1\right) \geq
\Gamma\left(\frac{3-p}{2}n-\frac{1}{2}+1\right) \geq C_p^n (n!)^{\frac{3-p}{2}},
\end{equation}
for some constant $C_p>0$ depending on $p$. For the last inequality, we used that fact that for any $a>0$ and $b\in \bR$, there exist some constants $c_{a,b}, C_{a,b}>0$ depending on $(a,b)$ such that 
\begin{equation}
\label{gam-prop}
c_{a,b}^n (n!)^a \leq \Gamma(an+b+1) \leq C_{a,b}^n (n!)^a \quad \mbox{for all} \quad n\geq 1.
\end{equation}
Relation \eqref{gam-prop} is a consequence of the Stirling formula:
\[
\Gamma(an+1) \sim (n!)^a a^{an} a^{1/2} (2\pi n)^{\frac{1-a}{2}} \quad \mbox{for any} \quad a>0,
\]
and the fact that 
\[
\Gamma(n+b) \sim \Gamma(n) n^b \quad \mbox{for any $b \in \bR$}.
\]

We use estimate \eqref{bound-Il}, combined with \eqref{Gam1}, \eqref{ineq-t} and \eqref{Gam2}. 
Consequently, 
\begin{equation}
\label{case1}
\cI_n(\pmb{\ell}) \leq \sqrt{\frac{p}{2}}  G_{t-r}^2(x-y) \frac{   \Gamma(\frac{3-p}{2})^{n+1} B_{T,p}^n (t-r)^{\frac{3-p}{2}n}}{C_p^n (n!)^{\frac{3-p}{2}}}, \quad \mbox{if $n>N_0$}.
\end{equation}

For $n\leq N_0$, we use the bound \eqref{bound-Il}, combined with \eqref{Gam1} and \eqref{ineq-t}. Hence,
\begin{align}
\nonumber
\mathcal{I}_n(\pmb{\ell})
&\leq \sqrt{\frac{p}{2}}G^2_{t-r}(x-y)\frac{  \Gamma(\frac{3-p}{2})^{n+1}  B^n_{T,p}(t-r)^{\frac{3-p}{2}n}}
{\Gamma(\sum_{i=0}^n\beta_i^{(\pmb{\ell})}+n+1)}
\\
\label{case12}
&\leq \sqrt{\frac{p}{2}} \Gamma(\frac{3-p}{2})H_{t-r}^{(p)}(x-y)C_0^n(t-r)^{\frac{3-p}{2}n}K_{N_0}, \quad \mbox{if $n\leq N_0$},
\end{align}
where $K_{N_0}^{-1} = \min _{1\leq n\leq N_0, \ell_1,...\ell_n \in \{1,2\}} \Gamma(\sum_{i=0}^n\beta_i^{(\pmb{\ell})}+n+1)$, 
and $C_0$ is a constant which depends on $(T,p)$ and is non-decreasing in $T$.

\medskip

{\em Case 2.} $\ell_i=2$ for all $i=1,\ldots,n$. Then $\beta_i^{(\pmb{\ell})}=\frac{1-p}{2}$ for all $i=1,\ldots,n$ and 
\[
\gamma^{(\pmb{\ell})}=\frac{t-t_1}{p}.
\]
Using the semigroup property, it follows that
\begin{align*}
& \bE\left[G_{t_1-r}^p(x-y-\sum_{i=1}^n X_i^{(\pmb{\ell})}) \right]= \int_\bR G_{t_1-r}^p(x-y-z)G_{\frac{t-t_1}{p}}(z) dz\\
& \quad =c_p (t_1-r)^{\frac{1-p}{2}} \int_\bR G_{ \frac{t_1-r}{p} }(x-y-z)G_{\frac{t-t_1}{p}}(z) dz=c_p (t_1-r)^{\frac{1-p}{2}} G_{\frac{t-r}{p}}(x-y)\\
& \quad  =(t_1-r)^{\frac{1-p}{2}}  (t-r)^{-\frac{1-p}{2}}  G_{t-r}^p(x-y).
\end{align*}
Hence, using Lemma \ref{beta-int} and property \eqref{gam-prop} of the Gamma function, we have for any $n\geq 1$,
\begin{align}
\nonumber
\cI_n(\pmb{\ell}) &= (t-r)^{-\frac{1-p}{2}} G_{t-r}^p(x-y) \int_{T_n(r,t)} \prod_{i=1}^{n} (t_{i+1}-t_i)^{\frac{1-p}{2}} (t_1-r)^{\frac{1-p}{2}} d\pmb{t}\\
\label{case2}
& = G_{t-r}^p(x-y) \frac{\Gamma\left(\frac{3-p}{2}\right)^{n+1} (t-r)^{\frac{3-p}{2}n}}{\Gamma\left(\frac{3-p}{2}n+\frac{1-p}{2}+1\right)} \leq  G_{t-r}^p(x-y) \frac{\Gamma\left(\frac{3-p}{2}\right)^{n+1} (t-r)^{\frac{3-p}{2}n}}{C_p^n (n!)^{\frac{3-p}{2}}}.
\end{align}

Summarizing \eqref{case1} and \eqref{case2}, for all $n>N_0$, we have: 
\begin{equation}
\label{sum-case12}
\cI_n(\pmb{\ell}) \leq H_{t-r}^{(p)}(x-y) \frac{C_0^n (t-r)^{\frac{3-p}{2}n}}{(n!)^{\frac{3-p}{2}}} \quad \mbox{for any $\pmb{\ell} \in \{1,2\}^n$}.
\end{equation}

Whereas for $n\leq N_0$, we use \eqref{case12} and \eqref{case2} to infer that
\[
\cI_n(\pmb{\ell}) \leq C_1^n H_{t-r}^{(p)}(x-y) (t-r)^{\frac{3-p}{2}n} \quad \mbox{for any $\pmb{\ell} \in \{1,2\}^n$}.
\]

Returning to \eqref{Snt}, it follows that 
\begin{align*}
S_n & \leq C^n H_{t-r}^{(p)}(x-y) \frac{(t-r)^{\frac{3-p}{2}n}}{(n!)^{\frac{3-p}{2}}} \quad \mbox{if $n>N_0$}, \\
S_n & \leq C^n H_{t-r}^{(p)}(x-y)  (t-r)^{\frac{3-p}{2}n}\quad \mbox{if $n\leq N_0$},
\end{align*}
where $C>0$ is a constant which depends on $(T,p)$ and is non-decreasing in $T$.

We now return to \eqref{Dun-bound}. We use the bound $G_{t-r}^p(x-y) \leq H_{t-r}^{(p)}(x-y)$ for the first term. 
It follows that for all $n \geq 1$,
\begin{align*}
& \bE|D_{r,y,z} u_{n}(t,x)|^p \\
&\leq 2^{p-1} |z|^p K_{T,p}^p H_{t-r}^{(p)}(x-y)\left\{
\sum_{k=0}^{N_0} (2^{p-1} A_{T,p})^k C^k(t-r)^{\frac{3-p}{2}k} +\sum_{k\geq N_0+1} (2^{p-1} A_{T,p})^k   \frac{C^k (t-r)^{\frac{3-p}{2}k}}{(k!)^{\frac{3-p}{2}}} \right\}\\
& \le C_{T,p}^p |z|^p H_{t-r}^{(p)}(x-y).
\end{align*}
This concludes the proof of \eqref{Du-G}.
\end{proof}

The next result gives the key estimate for the second Malliavin derivative.
 
\begin{theorem}
\label{key-th-D2u}
a) For any $t>0$ and $x \in \bR$, $u(t,x) \in {\rm dom}(D^2)$. 

b) Let $p\in [2,3)$ be such that $m_p<\infty$.
For any $0\leq t \leq T$, $x \in \bR$ and $\xi_1=(r_1,y_1,z_1),\xi_2=(r_2,y_2,z_2)\in \bf{Z}$ with $r_1,r_2 \in [0,t]$,
\begin{align*}
\| D^2_{\xi_1,\xi_2}u(t,x)  \|_p
\le 
\begin{cases}
C_{T,p}'|z_1z_2|g_{t-r_2}^{(p)}(x-y_2)g_{r_2-r_1}^{(p)}(y_2-y_1),
& \text{if } r_1<r_2, \\[0.3em]
C_{T,p}'|z_1z_2| g_{t-r_1}^{(p)}(x-y_1)g_{r_1-r_2}^{(p)}(y_1-y_2),
& \text{if } r_2<r_1.
\end{cases}
\end{align*}
\end{theorem}

\begin{proof} 
We proceed as in the proof of Theorem \ref{key-th-Du}. Note that $D^2 u_1=D^2 u_0=0$. For the same reasons as in the proof of Theorem \ref{key-th-Du}, for all $n\geq 0$, $u_n \in Dom(D^2)$. We will prove below that for any $p\in [2,3)$ such that $m_p<\infty$, $0\leq t \leq T$, $x \in \bR$ and $\xi_1=(r_1,y_1,z_1),\xi_2=(r_2,y_2,z_2)\in \bf{Z}$ with $0\leq r_1<r_2 \leq t$, and $n \geq 2$, 
\begin{equation}
\label{D2u-G}
\bE\left|D_{\xi_1,\xi_2}^2 u_n(t,x)\right|^p \leq (C_{T,p}')^p |z_1z_2|^p H_{t-r_2}^{(p)}(x-y_2)H_{r_2-r_1}^{(p)}(y_2-y_1),
\end{equation}
where $C_{T,p}'>0$ is a constant depending on $(T,p)$ which is non-decreasing in $T$. 
In particular, for $p=2$, this implies that
\[
\sup_{n\geq 2}\sup_{(t,x)\in [0,T] \times \bR}\bE \|D^2 u_n(t,x)\|_{\cH^{\otimes 2}}^2 <\infty.
\]

Indeed, by \eqref{D2u-G} and \eqref{int-H}, we obtain that for any $n \geq 2$, $t \in [0,T]$ and $x \in \bR$,
\begin{align*}
& \bE \|D^2 u_n(t,x)\|_{\cH^{\otimes 2}}^2 =\int_{{\bf Z}^2} \bE |D_{\xi_1,\xi_2}^2 u_n(t,x)|^2 \fm(d\xi_1) \fm(d\xi_2) \\
&\quad \leq 2 m_2^2 (C_{T,2}')^2 \int_{0<r_1<r_2<t} \int_{\bR^2}  H_{t-r_2}^{(2)}(x-y_2) H_{r_2-r_1}^{(2)}(y_2-y_1)  d\pmb{y}d\pmb{r}\\
& \quad \leq 8 m_2^2 (C_{T,2}')^2 c_2^2\int_{0<r_1<r_2<t} (t-r_2)^{-1/2} (r_2-r_1)^{-1/2}d\pmb{r}\\
& \quad =8\pi m_2^2(C_{T,2}')^2 c_2^2 \, t ,
\end{align*}
here $\pmb{r}=(r_1,r_2)$ and $\pmb{y}=(y_1,y_2)$.
We apply Lemma \ref{lem123} to $F_n=u_n(t,x)$ and $F=u(t,x)$ for fixed $(t,x)\in \bR_{+}\times \bR$. We conclude that
$u(t,x)\in {\rm dom}(D^2)$ and $\{D^2u_n(t,x)\}_{n\geq 1}$ converges to $D^2u(t,x)$ in the weak topology
$L^2(\Omega;\mathcal{H}^{\otimes  2 })$. This proves part a).

Taking power $1/p$ in \eqref{D2u-G} and using the inequality $(x+y)^a \leq x^a+y^a$ for any $x,y>0$ and $a\in (0,1)$, we infer that for any $\xi_1=(r_1,y_1,z_1),\xi_2=(r_2,y_2,z_2)\in \bf{Z}$ with $0\leq r_1<r_2 \leq t$, 
\[
\sup_{n\geq 2}\|D_{\xi_1,\xi_2}^2 u_n(t,x)\|_p \leq C_{T,p}' |z_1z_2| g_{t-r_2}^{(p)}(x-y_2)  g_{r_2-r_1}^{(p)}(y_2-y_1).
\]
Then, the conclusion in part b) follows by Lemma A.1 of \cite{BS26}.

\medskip

We now prove \eqref{D2u-G}. Similarly to the proof of Proposition 3.1 of \cite{BZ26}, it can be proved that for $\fm\times\fm$-almost $(\xi_1,\xi_2)\in{\bf Z}^2$,
$D_{\xi_1,\xi_2}^2u_n$ has a predictable modification (denoted also $D_{\xi_1,\xi_2}^2u_n$). We work with this modification. Using the  recurrence relation \eqref{Picard}, and Lemmas \ref{Heisenberg} and \ref{Ito-Skor}, we obtain: for any $\xi_1=(r_1,y_1,z_1),\xi_2=(r_2,y_2,z_2)\in \bf{Z}$ with $0\leq r_1<r_2 \leq t$, and $n\geq 1$,
 \begin{equation}
 \label{D2u}
D^2_{\xi_1,\xi_2}u_{n+1}(t,x)
= G_{t-r_2}(x-y_2)z_2D_{\xi_1}u_n(r_2,y_2)
+\int_{r_2}^t\int_{\mathbb{R}}
G_{t-s}(x-y)D^2_{\xi_1,\xi_2}u_n(s,y)L(ds,dy).
\end{equation}

Applying \eqref{estimate} with $\Phi=D^2u_n$, followed by \eqref{Du-pth-moment}, we obtain that for any
 $\xi_1=(r_1,y_1,z_1),\xi_2=(r_2,y_2,z_2)\in \bf{Z}$ with $0\leq r_1<r_2 \leq t$, and $n\geq 1$,
\begin{align*}
\bE|D^2_{\xi_1,\xi_2}u_{n+1}(t,x)|^p 
\le  2^{p-1} & \bigg\{
C_{T,p}^p  G_{t-r_2}^p(x-y_2)|z_1z_2|^p H_{r_2-r_1}^{(p)}(y_2-y_1)
  \\
& + A_{T,p} \int_{r_2}^t\int_{\bR} H_{t-s}^{(p)}(x-y)\bE|D^2_{\xi_1,\xi_2}u_{n}(s,y)|^p
dsdy \bigg\}.
\end{align*}

By induction, we infer that for any $n\ge 2$,
\begin{align*}
& \bE|D^2_{\xi_1,\xi_2}u_{n}(t,x)|^p \le \\
& \quad 2^{p-1}C_{T,p}^p |z_1z_2|^p H_{r_2-r_1}^{(p)}(y_2-y_1)\left\{ G_{t-r_2}^p(x-y_2)+ \sum_{k=1}^{n-2}\left(2^{p-1}A_{T,p}\right)^k S_k(r_2,t,x-y_2) \right\},
\end{align*}
where $S_n$ is given by \eqref{def-Sn}.
By the same argument as in the proof of Theorem \ref{key-th-Du}, we infer that
\begin{equation*}
 G_{t-r_2}^p(x-y_2)+ \sum_{k=1}^{n-2}\left(2^{p-1}A_{T,p}\right)^k S_k(r_2,t,x-y_2) \leq C_t H_{t-r_2}^{(p)}(x-y_2),
\end{equation*}
where $C_t>0$ is a constant that depends on $(t,p)$ and is non-decreasing in $t$. This concludes the proof of \eqref{D2u-G}.

\end{proof}

\section{Proofs of the main results}
\label{section-proofs}

In this section, we include the proofs of Theorems \ref{ergodic-cov}, \ref{QCLT} and \ref{FCLT}.

\medskip

For any $0\leq r <t$, $R>0$ and $y \in \bR$, we define:
\begin{equation}
\label{def-phi}
\varphi_{t,R}(r,y)=\int_{-R}^{R}G_{t-r}(x-y)dx \quad \mbox{and} \quad \varphi_{t,R}(y)=\varphi_{t,R}(0,y).
\end{equation}

\subsection{Proof of Theorem \ref{ergodic-cov}}

In this section, we give the proof of Theorem \ref{ergodic-cov}. 

\medskip

Part a) follows exactly as Theorem 1.1 of \cite{BZ24}. The stationarity is a consequence of the spatial homogeneity of the L\'evy white noise $L$, while ergodicity follows from Lemma 4.2 of \cite{BZ24}, using the key estimate \eqref{Du-pth-moment} for $Du$ with $p=2$. 

\medskip

For part b), we argue using the connection with the Gaussian model \eqref{pAm-Gauss}, whose solution is denoted by $U$. Similarly to \eqref{series2}, it can be proved that for any $t,s\geq 0$ and  $x,y \in \bR$, 
\begin{align*}
& \bE\big[\big(u(t,x)-1\big)\big(u(s,y)-1\big)\big]=\sum_{n\geq 1}n!\,\langle \widetilde{F}_n(\cdot,t,x),\widetilde{F}_n(\cdot,s,y)\rangle_{\cH^{\otimes n}}\\
&  \quad = \sum_{n\geq 1}n!\, m_2^n \, \langle \widetilde{f}_n(\cdot,t,x),\widetilde{f}_n(\cdot,s,y)\rangle_{L^2((\bR_{+} \times \bR)^n)}=\bE\big[\big(U(t,x)-1\big)\big(U(s,y)-1\big)\big],
\end{align*}
and hence
\begin{align*}
\bE\big[ F_R(t) F_R(s)\big] & =\int_{-R}^R \int_{-R}^R \bE\big[\big(u(t,x)-1\big)\big(u(s,y)-1\big)\big] dxdy =
\bE\big[ H_R(t) H_R(s)\big],
\end{align*}
where 
\[
H_R(t):=\int_{-R}^R \big(U(t,x)-1\big)dx=\sqrt{m_2}\int_{0}^{t}\int_{\bR} \varphi_{t,R}(r,y) U(r,y)W(dr,dy).
\]
For the second equality above, we used the fact that
\[
U(t,x)=1+\sqrt{m_2}\int_0^t \int_{\bR} G_{t-r}(x-z) U(r,z) W(dr,dz),
\]
and definition \eqref{def-phi} of the function $\varphi_{t,R}$. By It\^o isometry, and the stationarity of $\{U(t,x)\}_{x\in \bR}$,
\[
\bE[H_R(t)H_R(s)] = m_2\int_0^{t\wedge s} \bE|U(r,0)|^2 \int_{\bR}\varphi_{t,R}(r,z)\varphi_{s,R}(r,z) dz dr.
\]

Note that for any $t>0$ and $R>0$, 
\begin{equation}
\label{F-phi}
\cF \varphi_{t,R}(\xi)=\cF G_{t}(\xi) \cF 1_{[-R,R]} (\xi)=2
e^{-\frac{t|\xi|^2}{2}} \frac{\sin(R|\xi|)}{|\xi|},
\end{equation}
where $\cF \varphi(\xi)=\int_{\bR^d}e^{-i \xi x} \varphi(x)dx$ is the Fourier transform of $\varphi$.

By Plancherel's theorem, \eqref{F-phi} and the change of variable $\eta=R\xi$:
\begin{align*}
\frac{1}{R}\int_{\bR}\varphi_{t,R}(y)\varphi_{s,R}(y)dy=\frac{2}{\pi R} \int_{\bR} e^{-\frac{(t+s)|\xi|^2}{2}} \frac{\sin^2(R|\xi|)}{|\xi|^2} d\xi=\frac{2}{\pi} \int_{\bR} e^{-\frac{(t+s)|\eta|^2}{2R^2}} \frac{\sin^2(|\eta|)}{|\eta|^2} d\eta \leq 2.
\end{align*}
By the dominated convergence theorem,
\[
\lim_{R \to \infty}\frac{1}{R}\int_{\bR}\varphi_{t,R}(y)\varphi_{s,R}(y)dy=\frac{2}{\pi} \int_{\bR} \frac{\sin^2(|\eta|)}{|\eta|^2} d\eta =2,
\]
and
\[
\lim_{R \to \infty}\frac{1}{R}\bE[H_R(t)H_R(s)] = 2m_2\int_0^{t\wedge s} \bE|U(r,0)|^2 dr=:\Sigma_{t,s}.
\]
The final expression of $\Sigma_{t,s}$ is obtained using the explicit form of $E|U(t,0)|^2$:
\begin{align*}
\bE|U(t,0)|^2 &=1+\sum_{n\geq 1} m_2^n \int_{T_n(t)} \int_{\bR^n} \prod_{i=1}^{n}G_{t_{i+1}-t_i}^2(x_{i+1}-x_i) d\pmb{x} d\pmb{t}\\
&=\sum_{n\geq 0} \frac{\big(m_2^2 \, t/4\big)^{n/2}}{\Gamma(n/2+1)}=2 e^{m_2^2 t/4} \Phi\left(m_2 \sqrt{\frac{t}{2}} \right).
\end{align*}
To compute the integral above, we used \eqref{Gtp} and Lemma \ref{beta-int}, whereas for the last equality, we used the identity:
\[ 
\sum_{n\geq 0}\frac{x^n}{\Gamma(n/2+1)}=2e^{x^2} \Phi(x\sqrt{2}) \quad \mbox{for any}  \quad x>0.
\]

\subsection{Proof of Theorem \ref{QCLT}}

In this section, we give the proof of Theorem \ref{QCLT}. We will use the following result.

\begin{proposition}[Theorem 3.4 of \cite{trauthwein25}]
\label{tara}
Let $F\in {\rm dom}(D)$ with $\bE[F]=0$ and  ${\rm Var}(F)=\sigma^2 > 0$. Then, 
for any $p,q\in [1,2]$,
\begin{align*}
d_{\rm FM}\left( \frac{F}{\sigma}, Z\right) & \leq d_{\rm W}\left(\frac{F}{\sigma}, Z\right)  \leq    \gamma_1 +  \gamma_2 +  \gamma_3\\
  d_{\rm K}\left(\frac{F}{\sigma}, Z\right) & \leq  \sqrt{ \frac{\pi}{2}} (  \gamma_1 +  \gamma_2)   +  \gamma_4 +  \gamma_5+  \gamma_6 +  \gamma_7,
\end{align*}
where  $Z\sim N(0,1)$ and the seven quantities $ \gamma_1, \ldots,  \gamma_7$ are given as follows:
\begin{align}
\begin{aligned}
 \gamma_1&:=  \frac{2^{\frac2p + \frac12} }{\sqrt{\pi}}
 \sigma^{-2} \bigg( \int_{Z}  \bigg[ \int_{\bf Z}  \| D_{\xi_2}F \|_{2p}  \| D^2_{\xi_1,\xi_2}F\|_{2p} 
  \,  \fm(d\xi_2) \bigg]^p
  \fm(d\xi_1) \bigg)^{\frac1p}
\\
 \gamma_2&:=\frac{2^{\frac2p - \frac12} }{\sqrt{\pi}}
\sigma^{-2} \bigg( \int_{Z}  \bigg[ \int_{\bf Z}   \| D^2_{\xi_1,\xi_2}F\|^2_{2p}  \,  \fm(d\xi_2) \bigg]^p
  \fm(d\xi_1) \bigg)^{\frac1p}
\\
 \gamma_3&:= 2 \sigma^{-(q+1)}\int_{\bf Z} \| D_\xi F \|_{q+1}^{q+1}  \, \fm(d\xi)
\\
 \gamma_4&:= 2^{\frac{2}{p}} \sigma^{-2} \bigg( \int_{\bf Z} \| D_\xi F \|_{2p}^{2p}  \, \fm(d\xi) \bigg)^{\frac1p}
\\
 \gamma_5&:= (4p)^{\frac1p} \sigma^{-2} \bigg( \int_{Z^2} \| D^2_{\xi_1,\xi_2}  F \|_{2p}^{2p}
\, \fm(d\xi_1)\fm(d\xi_2)  \bigg)^{\frac1p}
\\
 \gamma_6&:=(2^{2+p}p)^{\frac1p} \sigma^{-2} \bigg( \int_{Z^2} \| D^2_{\xi_1,\xi_2}  F \|_{2p}^{p}
 \| D_{\xi_1}F \|_{2p}^p
\, \fm(d\xi_1)\fm(d\xi_2)  \bigg)^{\frac1p}\\
 \gamma_7&:=  \frac{  (8p)^{\frac1p} }{\sigma^2}
 \bigg( \int_{Z^2} \| D^2_{\xi_1,\xi_2}  F \|_{2p}
 \| D_{\xi_1}F \|_{2p}  \| D_{\xi_2}F \|_{2p}^{2(p-1)}
\, \fm(d\xi_1) \fm(d\xi_2)  \bigg)^{\frac1p}.
\end{aligned}
\label{gamma17}
\end{align}
\end{proposition}

We first collect some estimates which will be used repeatedly in the bounds for
$\gamma_1,\ldots,\gamma_7$. Throughout this subsection, the implicit constants
in $\lesssim$ may depend on $T$ and $p$, but not on $R$.

For any $t>0,R>0,p>0$ and $y \in \bR$, we define
\[
K_t^{(p)}(y)=G_t(y)+G_{pt}(y) \quad \mbox{and} \quad
\Phi_{t,R}^{(p)}(y)=\varphi_{t,R}(y)+\varphi_{pt,R}(y).
\]

\begin{lemma}
\label{lem:FR-derivative-bounds}
Let $p\in[1,\frac{3}{2})$ be such that $m_{2p}<\infty$.

\textup{(i)} For any $0\le r<t\le T$, $y\in\bR$ and $z\in\mathbb{R}_0$,
\begin{equation}
\label{DFR-bound}
\big\lVert D_{r,y,z}F_R(t)\big\rVert_{2p}
\lesssim
|z|\Phi_{t-r,R}^{(p)}(y).
\end{equation}

\textup{(ii)} Let $\xi_i=(r_i,y_i,z_i)$, $i=1,2$. For any
$0\le r_1,r_2<t\le T$ with $r_1\neq r_2$, $y_1,y_2\in\bR$ and
$z_1,z_2\in\mathbb{R}_0$, we have
\begin{equation}
\label{D2FR-bound}
\big\lVert D_{\xi_1,\xi_2}^2F_R(t)\big\rVert_{2p}
\lesssim
|z_1z_2|
\begin{cases}
K_{r_2-r_1}^{(p)}(y_2-y_1)\Phi^{(p)}_{t-r_2,R}(y_2),
& \text{if } r_1<r_2,\\
K^{(p)}_{r_1-r_2}(y_1-y_2)\Phi^{(p)}_{t-r_1,R}(y_1),
& \text{if } r_2<r_1.
\end{cases}
\end{equation}
\end{lemma}

\begin{proof}
 By Minkowski's inequality, \eqref{Du-pth-moment} with exponent $2p$ and \eqref{Gtp},
 \begin{align*}
\big\lVert D_{r,y,z} F_R(t)\big\rVert_{2p}
&\leq
\int_{-R}^{R}
\big\lVert D_{r,y,z} u(t,x)\big\rVert_{2p}dx
\\
&\lesssim
|z|
\int_{-R}^{R}
\left(
G_{t-r}(x-y)+G^{\frac{1}{p}}_{t-r}(x-y)
\right)dx\\
&\lesssim |z|
\int_{-R}^{R}
\left(
G_{t-r}(x-y)+G_{p(t-r)}(x-y)
\right)dx
\end{align*}
This proves \eqref{DFR-bound}. We now prove \eqref{D2FR-bound}. We only consider the case $r_1<r_2$, since the case $r_2<r_1$ follows by symmetry. By the same idea,
\begin{align*}
\big\lVert D^2_{\xi_1,\xi_2}F_R(t)\big\rVert_{2p}
&\leq
\int_{-R}^{R}
\big\lVert D^2_{\xi_1,\xi_2}u(t,x)\big\rVert_{2p}\,dx
\\
&\lesssim
|z_1z_2|\,
K^{(p)}_{r_2-r_1}(y_2-y_1)
\int_{-R}^{R}K^{(p)}_{t-r_2}(x-y_2)\,dx.
\end{align*}
\end{proof}

Note that for any $\alpha>0$, $p>0$, $t>0$ and $y \in \bR$,
\begin{equation}
\label{bound-K}
[K_{t}^{(p)}(y)]^{\alpha} \le C_{\alpha,p} t^{\frac{1-\alpha}{2}} K_{\frac{t}{\alpha}}^{(p)}(y),
\end{equation}
where $C_{\alpha,p}>0$ is a constant depending on $(\alpha,p)$.
Indeed, using inequality $(a+b)^\alpha \leq C_{\alpha}(a^\alpha+b^\alpha)$ for any $a,b>0$ with $C_{\alpha}=2^{\alpha -1} \vee 1$, and the first relation in \eqref{Gtp}, we obtain
\[
[K_t^{(p)}(y)]^{\alpha}
\leq C_{\alpha}
\Big(G_t^\alpha(y)+G_{pt}^\alpha(y)\big)
= C_{\alpha} c_{\alpha}(1+p^{\frac{1-\alpha}{2}})
t^{\frac{1-\alpha}{2}}
\big(
G_{\frac{t}{\alpha}}(y)+G_{\frac{pt}{\alpha}}(y)
\big).
\]

The following lemma has a trivial proof. We omit the details.

\begin{lemma}
For every $\alpha\geq 1$, $t>0$, $R>0$ and $p>0$,
\begin{align}
\label{K-bound}
& \int_{\bR}[K_t^{(p)}(y)]^\alpha dy \leq C_{\alpha,p} t^{\frac{1-\alpha}{2}}, \\
\label{varphi-bound} 
&\int_{\bR}\varphi_{t,R}^\alpha(y)dy \leq 2 R, \\
\label{Phi-bound}
& \int_{\bR}[\Phi_{t,R}^{(p)}(y)]^{\alpha} dy \leq 2^{\alpha+1}  R,
\end{align}   
where $C_{\alpha,p}$ is the constant from \eqref{bound-K}.
\end{lemma}

\begin{lemma}
\textup{(Past convolution bound.)}
For every $1\leq \alpha<2$, $\beta\geq 1$, $0\leq r<t\leq T$ and
$y\in\mathbb R$, we have
\begin{equation}
\label{past}
\int_0^r\int_{\mathbb R}
[K^{(p)}_{r-s}(y-y')]^\alpha
[\Phi^{(p)}_{t-s,R}(y')]^\beta
dy'ds
\lesssim
\varphi_{p(t+r),R}(y).
\end{equation}
\end{lemma}

\begin{proof}
Since, $0\leq\Phi^{(p)}_{t,R}(y)\leq 2$, we have $[\Phi^{(p)}_{t,R}(y)]^\beta\lesssim \Phi^{(p)}_{t,R}(y)$. Therefore, using \eqref{bound-K},
\begin{align*}
&\int_0^r\int_{\mathbb R}
[K_{r-s}^{(p)}(y-y')
]^\alpha[\Phi_{t-s,R}^{(p)}(y')]^\beta
dy'ds
\\
&\lesssim
\int_0^r\int_{\mathbb R}
(r-s)^{-\frac{\alpha-1}{2}}
K^{(p)}_{\frac{r-s}{\alpha}}(y-y')
\Phi^{(p)}_{t-s,R}(y')
dy'ds
\\
&=
\sum_{c,d\in\{1,p\}}
\int_0^r
(r-s)^{-\frac{\alpha-1}{2}}
\int_{\mathbb R}
G_{\frac{c(r-s)}{\alpha}}(y-y')
\varphi_{d(t-s),R}(y')
dy'ds .
\end{align*}

By the semigroup property of the heat kernel,
\[
\int_{\mathbb R}
G_{\frac{c(r-s)}{\alpha}}(y-y')
\varphi_{d(t-s),R}(y')
dy'
=
\varphi_{\frac{c(r-s)}{\alpha}+d(t-s),R}(y).
\]
Since $0\leq s\leq r<t$ and $c,d\in\{1,p\}$,
\[
\frac{1}{\alpha}(r-s) \leq \frac{c(r-s)}{\alpha}+d(t-s)
\leq
p(t+r).
\]
Thus, by the heat-kernel comparison estimate \eqref{sqrt(t over s)},
 \[
\varphi_{\frac{c(r-s)}{\alpha}+d(t-s),R}(y)
\lesssim
\sqrt{
\frac{p(t+r)}
{\frac{c(r-s)}{\alpha}+d(t-s)}
}
\varphi_{p(t+r),R}(y)  \lesssim (r-s)^{-1/2}\varphi_{p(t+r),R}(y).
\]

Therefore,
\begin{align*}
&\int_0^r\int_{\mathbb R}
[K_{r-s}^{(p)}(y-y')]^\alpha
[\Phi^{(p)}_{t-s,R}(y')]^\beta
dy'ds \lesssim \varphi_{p(t+r),R}(y) \int_0^r
(r-s)^{-\frac{\alpha}{2}}
ds \lesssim \varphi_{p(t+r),R}(y).
\end{align*}
\end{proof}

\begin{lemma}
\textup{(Future convolution bound.)}
For every $1\leq \alpha<3$, $\beta\geq 1$, $0\leq r<t\leq T$ and
$y\in\mathbb R$, we have
\begin{equation}
\label{future}
\int_r^t\int_{\mathbb R}
[K^{(p)}_{s-r}(y'-y)]^\alpha
[\Phi^{(p)}_{t-s,R}(y')]^\beta
dy'ds
\lesssim
\varphi_{p(t-r),R}(y).
\end{equation}
\end{lemma}

\begin{proof}
Using the same arguments as for the proof of \eqref{past}, we have
\begin{align*}
&\int_r^t\int_{\mathbb R}
[K^{(p)}_{s-r}(y'-y)]^\alpha
[\Phi^{(p)}_{t-s,R}(y')]^\beta
dy'ds
\\
&\lesssim \sum_{c,d\in\{1,p\}}
\int_r^t
(s-r)^{-\frac{\alpha-1}{2}}
\varphi_{\frac{c(s-r)}{\alpha}+d(t-s),R}(y)ds.
\end{align*}

Since $r\leq s\leq t$ and $c,d\in\{1,p\}$, we have
\[
\frac{1}{\alpha}(t-r)
\leq
\frac{c(s-r)}{\alpha}+d(t-s)
\leq
p(t-r).
\]
Thus, by the heat-kernel comparison estimate \eqref{sqrt(t over s)},
\[
\varphi_{\frac{c(s-r)}{\alpha}+d(t-s),R}(y)
\lesssim
\varphi_{p(t-r),R}(y).
\]
Therefore, 
\begin{align*}
&\int_r^t\int_{\mathbb R}
[K^{(p)}_{s-r}(y'-y)]^\alpha
[\Phi^{(p)}_{t-s,R}(y')]^\beta
dy'ds \lesssim
\varphi_{p(t-r),R}(y)
\int_r^t
(s-r)^{-\frac{\alpha-1}{2}}
ds \lesssim
\varphi_{p(t-r),R}(y).
\end{align*}
\end{proof}

{\em Proof of Theorem \ref{QCLT}:} We apply Proposition \ref{tara} to $F=F_{R}(t)$. 

We recall that $p \in [1,\frac{3}{2})$ is such that $m_{p}<\infty$ and $m_{2p}<\infty$.

For $\gamma_1,\gamma_2,\gamma_5,\gamma_6$ and $\gamma_7$, each integral over
$Z^2$ will be split into $\{r_2<r_1\}$ and $\{r_2>r_1\}$. This is necessary because the bound for $\lVert D^2_{\xi_1,\xi_2}F_R(t)\rVert_{2p}$ depends on whether $r_1<r_2$ or
$r_2<r_1$. The diagonal case $r_1=r_2$ has zero measure and can be ignored.

\medskip

\textbf{$\bullet$ Estimation of $\gamma_1$.}

\begin{equation}
\label{gamma1}
\gamma_1^p \lesssim R^{-p} \int_{\bf Z}
\left(
I_{1,1}(\xi_1)+I_{1,2}(\xi_1)
\right)^p
\fm(d\xi_1),    
\end{equation}
where
\begin{align*}
I_{1,1}(\xi_1)&:=
\int_{\bf Z} 1_{\{r_2<r_1\}}  
\big\lVert D_{\xi_2}F_R(t)\big\rVert_{2p}
\big\lVert D_{\xi_1,\xi_2}^2F_R(t)\big\rVert_{2p}
\fm(d\xi_2)\\   
&\lesssim
|z_1|m_2
\Phi^{(p)}_{t-r_1,R}(y_1) \int_0^{r_1}
\int_{\mathbb R}
K^{(p)}_{r_1-r_2}(y_1-y_2)
\Phi^{(p)}_{t-r_2,R}(y_2)
dy_2dr_2
\\
&\lesssim
|z_1|m_2
\Phi^{(p)}_{t-r_1,R}(y_1)
\varphi_{p(t+r_1),R}(y_1),
\end{align*}
and
\begin{align*}
I_{1,2}(\xi_1)&:=
\int_{\bf Z} 1_{\{r_1<r_2\}} 
\big\lVert D_{\xi_2}F_R(t)\big\rVert_{2p}
\big\lVert D_{\xi_1,\xi_2}^2F_R(t)\big\rVert_{2p}
\fm(d\xi_2)\\
&\lesssim
|z_1|m_2
\int_{r_1}^t\int_{\mathbb R}
K^{(p)}_{r_2-r_1}(y_2-y_1)
[\Phi^{(p)}_{t-r_2,R}(y_2)]^2
dy_2dr_2
\\
&\lesssim
|z_1|m_2
\varphi_{p(t-r_1),R}(y_1),
\end{align*}
by using of \eqref{DFR-bound}, \eqref{D2FR-bound}, \eqref{past} and \eqref{future}. Therefore,
\begin{equation*}
\gamma_1^p
\lesssim
R^{-p}m_2^p m_p \left(
\int_0^t\int_{\mathbb R}
[\Phi^{(p)}_{t-r_1,R}(y_1)]^p
\varphi_{p(t+r_1),R}^p(y_1)
dy_1dr_1+ \int_0^t\int_{\mathbb R}
\varphi_{p(t-r_1),R}^p(y_1)
dy_1dr_1 \right)   
\end{equation*}
Since $0\leq\varphi_{R,a}\leq 1$, \eqref{varphi-bound} and \eqref{Phi-bound}, we obtain $\gamma_1^p \lesssim R^{-(p-1)}$.

\textbf{$\bullet$  Estimation of $\gamma_2$.}
\begin{equation}
\label{gamma2}
\gamma_2^p
\lesssim
R^{-p}
\int_{\bf Z}
\left(
I_{2,1}(\xi_1)+I_{2,2}(\xi_1)
\right)^p
\fm(d\xi_1),
\end{equation}
where
\begin{align*}
I_{2,1}(\xi_1)&:=
\int_{\bf Z} 1_{\{r_2<r_1\}} 
\big\lVert D_{\xi_1,\xi_2}^2F_R(t)\big\rVert_{2p}^2
\fm(d\xi_2)\\
&\lesssim
|z_1|^2m_2
[\Phi^{(p)}_{t-r_1,R}(y_1)]^2
\int_0^{r_1}\int_{\mathbb R}
[K^{(p)}_{r_1-r_2}(y_1-y_2)]^2
dy_2dr_2
\\
&\lesssim
|z_1|^2m_2
[\Phi^{(p)}_{t-r_1,R}(y_1)]^2
\int_0^{r_1}
(r_1-r_2)^{-\frac12}
dr_2
\\
&\lesssim
|z_1|^2m_2
[\Phi^{(p)}_{t-r_1,R}(y_1)]^2,
\end{align*}
and
\begin{align*}
I_{2,2}(\xi_1)&:=
\int_{\bf Z} 1_{\{r_1<r_2\}} 
\big\lVert D_{\xi_1,\xi_2}^2F_R(t)\big\rVert_{2p}^2
\fm(d\xi_2) \\ 
&\lesssim
|z_1|^2m_2
\int_{r_1}^t\int_{\mathbb R}
[K^{(p)}_{r_2-r_1}(y_2-y_1)]^2
[\Phi^{(p)}_{t-r_2,R}(y_2)]^2
dy_2dr_2
\\
&\lesssim
|z_1|^2m_2
\varphi_{p(t-r_1),R}(y_1).
\end{align*}
by using of \eqref{DFR-bound}, \eqref{D2FR-bound}, \eqref{past} and \eqref{future}. Therefore,
\begin{align*}
\gamma_2^p&\lesssim
R^{-p}m_2^p m_{2p}
\left(
\int_0^t\int_{\mathbb R}
[\Phi^{(p)}_{t-r_1,R}(y_1)]^{2p}
dy_1dr_1
+
\int_0^t\int_{\mathbb R}
\varphi_{p(t-r_1),R}^p(y_1)
dy_1dr_1
\right).  
\end{align*}
By \eqref{varphi-bound} and \eqref{Phi-bound}, we obtain $\gamma_2^p \lesssim R^{-(p-1)}$.

\medskip
\textbf{$\bullet$  Estimation of $\gamma_3$.}
Choose $q=2p-1$. Then $q\in(1,2)$ and $q+1=2p$. 
\begin{align*}
\gamma_3
&\lesssim
R^{-p}
\int_{\bf Z}
\big\lVert D_{\xi}F_R(t)\big\rVert_{2p}^{2p}
\fm(d\xi)\lesssim
R^{-p}m_{2p}
\int_0^t\int_{\mathbb R}
[\Phi^{(p)}_{t-r,R}(y)]^{2p}
dydr
\end{align*}
By \eqref{Phi-bound}, $\gamma_3 \lesssim R^{-(p-1)}\lesssim R^{-\left(1-\frac1p\right)}$ for $R>1$.

\medskip
\textbf{$\bullet$  Estimation of $\gamma_4$.}
\begin{align*}
\gamma_4^p
&\lesssim
R^{-p} \int_{\bf Z} \big\lVert D_{\xi}F_R(t)\big\rVert_{2p}^{2p}
\fm(d\xi)\lesssim
R^{-p}m_{2p}
\int_0^t\int_{\mathbb R}
[\Phi^{(p)}_{t-r,R}(y)]^{2p}
dydr.
\end{align*}
By \eqref{Phi-bound}, $\gamma_4^p \lesssim R^{-(p-1)}$.

\medskip

\textbf{$\bullet$  Estimation of $\gamma_5$.}
\begin{align*}
\gamma_5^p \lesssim R^{-p}m_{2p}^2
\left(
I_{5,1}+I_{5,2}
\right),
\end{align*}
where
\[
I_{5,1}:=
\int_0^t\int_0^{r_1}
\int_{\mathbb R^2}
[K^{(p)}_{r_1-r_2}(y_1-y_2)]^{2p}
[\Phi^{(p)}_{t-r_1,R}(y_1)]^{2p}
dy_1dy_2dr_2dr_1,
\]
and
\[
I_{5,2}:=
\int_0^t\int_{r_1}^{t}
\int_{\mathbb R^2}
[K^{(p)}_{r_2-r_1}(y_2-y_1)]^{2p}
[\Phi^{(p)}_{t-r_2,R}(y_2)]^{2p}
dy_1dy_2dr_2dr_1.
\]
By symmetry, it is enough to estimate $I_{5,1}$. Using \eqref{K-bound} with
$\alpha=2p$ and \eqref{Phi-bound},
\begin{align*}
I_{5,1}
&\lesssim  \int_0^t\int_0^{r_1} \int_{\bR} \left(\int_{\bR} [K^{(p)}_{r_1-r_2}(y_1-y_2)]^{2p} dy_2 \right) [\Phi^{(p)}_{t-r_1,R}(y_1)]^{2p}
dy_1dr_2dr_1\\
&\lesssim
\int_0^t\int_0^{r_1}
(r_1-r_2)^{-\frac{2p-1}{2}}
\left(
\int_{\mathbb R}
[\Phi^{(p)}_{t-r_1,R}(y_1)]^{2p}
dy_1
\right)
dr_2dr_1
\\
&\lesssim
R.
\end{align*}
The last time integral is finite because $p<\frac32$. Hence $I_{5,1}+I_{5,2}\lesssim R$ and $\gamma_5^p
\lesssim R^{-(p-1)}$.

\medskip

\textbf{$\bullet$ Estimation of $\gamma_6$.}
\begin{equation}
    \label{gamma6}
\gamma_6^p
\lesssim
R^{-p}m_{2p}m_p
\left(
I_{6,1}+I_{6,2}
\right),    
\end{equation}
where
\begin{align*}
  I_{6,1}&:=
\int_0^t\int_0^{r_1}
\int_{\mathbb R^2}
[K^{(p)}_{r_1-r_2}(y_1-y_2)]^{p}
[\Phi^{(p)}_{t-r_1,R}(y_1)]^{2p}
dy_1dy_2dr_2dr_1\\
&\lesssim
\int_0^t\int_0^{r_1}
(r_1-r_2)^{-\frac{p-1}{2}}
\left(
\int_{\mathbb R}
[\Phi^{(p)}_{t-r_1,R}(y_1)]^{2p}
dy_1
\right)
dr_2dr_1\\
&\lesssim R,
\end{align*}
and
\begin{align*}
I_{6,2}&:=
\int_0^t\int_{r_1}^{t}
\int_{\mathbb R^2}
[K^{(p)}_{r_2-r_1}(y_2-y_1)]^p
[\Phi^{(p)}_{t-r_2,R}(y_2)]^p
[\Phi^{(p)}_{t-r_1,R}(y_1)]^p
dy_1dy_2dr_2dr_1  \\
&=
\int_0^t\int_{\mathbb R}
[\Phi^{(p)}_{t-r_1,R}(y_1)]^p
\left(
\int_{r_1}^{t}\int_{\mathbb R}
[K^{(p)}_{r_2-r_1}(y_2-y_1)]^p
[\Phi^{(p)}_{t-r_2,R}(y_2)]^p
dy_2dr_2
\right)
dy_1dr_1
\\
&\lesssim
\int_0^t\int_{\mathbb R}
[\Phi^{(p)}_{t-r_1,R}(y_1)]^p
\varphi_{p(t-r_1),R}(y_1)
dy_1dr_1
\\
&\lesssim R,
\end{align*}
by using \eqref{future} with $\alpha=p$ and $\beta=p$, $0\leq\varphi_{R,a}\leq 1$ and
\eqref{Phi-bound}. Hence $I_{6,1}+I_{6,2}\lesssim R$ and $\gamma_6^p
\lesssim R^{-(p-1)}$.

\medskip
\textbf{$\bullet$  Estimation of $\gamma_7$.}
\begin{equation}
    \label{gamma7}
    \gamma_7^p
\lesssim
R^{-p}m_2m_{2p-1}
\left(
I_{7,1}+ I_{7,2}
\right),
\end{equation}
where
\begin{align*}
I_{7,1}
&=
\int_0^t\int_{\mathbb R}
[\Phi^{(p)}_{t-r_2,R}(y_2)]^{2p-2}
\left(
\int_{r_2}^{t}\int_{\mathbb R}
K^{(p)}_{r_1-r_2}(y_1-y_2)
[\Phi^{(p)}_{t-r_1,R}(y_1)]^2
\,dy_1\,dr_1
\right)
dy_2\,dr_2 \\
&\lesssim
\int_0^t\int_{\mathbb R}
[\Phi^{(p)}_{t-r_2,R}(y_2)]^{2p-2}
\varphi_{p(t-r_2),R}(y_2)
\,dy_2\,dr_2 \\
&\leq
\int_0^t\int_{\mathbb R}
[\Phi^{(p)}_{t-r_2,R}(y_2)]^{2p-1}
\,dy_2\,dr_2 \\
&\lesssim R,
\end{align*}
and
\begin{align*}
I_{7,2}&:=
\int_0^t\int_{r_1}^{t}
\int_{\mathbb R^2}
K^{(p)}_{r_2-r_1}(y_2-y_1)
\Phi^{(p)}_{t-r_1,R}(y_1)
[\Phi^{(p)}_{t-r_2,R}(y_2)]^{2p-1}
dy_1dy_2dr_2dr_1\\
&\lesssim
\int_0^t\int_{\mathbb R}
\Phi^{(p)}_{t-r_1,R}(y_1)
\left(
\int_{r_1}^{t}\int_{\mathbb R}
K^{(p)}_{r_2-r_1}(y_2-y_1)
[\Phi^{(p)}_{t-r_2,R}(y_2)]^{2p-1}
dy_2dr_2
\right)
dy_1dr_1
\\
&\lesssim
\int_0^t\int_{\mathbb R}
\Phi^{(p)}_{t-r_1,R}(y_1)
\varphi_{p(t-r_1),R}(y_1)
dy_1dr_1
\\
&\lesssim R.
\end{align*}
For the above estimate of  $I_{7,1}$,  we used \eqref{future} with $\alpha=1$ and $\beta=2$, together with the fact that
$\varphi_{p(t-r_2),R}\leq\Phi^{(p)}_{t-r_2,R}.$. Then we used \eqref{Phi-bound} with $\alpha=2p-1 \in [1,2)$. For $I_{7,2}$,
we used \eqref{past} and \eqref{future}. Hence $I_{7,1}+I_{7,2}\lesssim R$ and $\gamma_7^p
\lesssim R^{-(p-1)}$.

\medskip

Combining the above estimates, we obtain
\[
\gamma_i
\lesssim
R^{-\left(1-\frac1p\right)},
\qquad i=1,\ldots,7.
\]

\subsection{Proof of Theorem \ref{FCLT}}

In this section, we give the proof of Theorem~\ref{FCLT}.

We first give some preliminary estimates.

\begin{lemma}
\label{fclt-lemma1}
For any $R>0$ and $0\leq s<t$,
\[
\int_0^s\int_{\bR}
\left|
\varphi_{t,R}(r,y)-\varphi_{s,R}(r,y) \right|^2dydr \leq C_1 (t-s)^{3/2}, 
\]
where $C_1=\frac{2}{\pi} \int_{\bR}\frac{(1-e^{-\eta^2/2})^2}{\eta^4}d\eta$.
\end{lemma}

\begin{proof} By Plancherel's theorem and \eqref{F-phi},
\begin{align*}
 \int_0^s\int_{\bR} \left|   \varphi_{t,R}(r,y)-\varphi_{s,R}(r,y)
\right|^2dydr & =\frac{2}{\pi} \int_0^s\int_{\bR} \left(e^{-(t-r)\xi^2/2}
-e^{-(s-r)\xi^2/2}\right)^2 \frac{\sin^2(R|\xi|)}{|\xi|^2} d\xi dr\\
& \leq \frac{2}{\pi}\int_{\bR}
 \frac{(1-e^{-(t-s)\xi^2/2})^2}{|\xi|^2}
\left(\int_0^s e^{-(s-r)\xi^2}dr\right)d\xi \\
& \leq \frac{2}{\pi} \int_{\bR}
 \frac{(1-e^{-(t-s)\xi^2/2})^2}{|\xi|^4}d\xi\\
 &=\frac{2(t-s)^{3/2}}{\pi} \int_{\bR}\frac{(1-e^{-\eta^2/2})^2}{|\eta|^4} d\eta,
\end{align*}
where for the last equality, we used the change of variable $\eta=\sqrt{t-s}\xi$.
The last integral is finite since the integrand is bounded near $0$ (since $1-e^{-x} \leq x$ for $x>0$), and behaves like $|\eta|^{-4}$ at infinity. 
\end{proof}

\begin{lemma}
\label{fclt-lemma2}
For every $R>0$ and $t>0$,
\[
\|\varphi_{t,R}-\mathbf 1_{[-R,R]}\|_{L^2(\bR)}^2 \leq C_2 t^{1/2},
\]
where $C_2=\frac{2}{\pi} \int_{\bR}\frac{(1-e^{-\eta^2/2})^2}{\eta^2}d\eta$.
\end{lemma}

\begin{proof}
By Plancherel's theorem and \eqref{F-phi},
\begin{align*}
\int_{\bR} \left|   \varphi_{t,R}(y)-1_{[-R,R]}(y)\right|^2dy & =\frac{2}{\pi}\int_{\bR} \Big(1-e^{-\frac{t|\xi|^2}{2}}\Big)^2 \frac{\sin^2(R|\xi|)}{|\xi|^2} d\xi \\
& \leq \frac{2}{\pi} \int_{\bR}
 \frac{(1-e^{-t|\xi|^2/2})^2}{|\xi|^2} d\xi
 \\
 &=\frac{2t^{1/2}}{\pi} \int_{\bR}\frac{(1-e^{-\eta^2/2})^2}{|\eta|^2} d\eta,
\end{align*}
where for the last equality, we used the change of variable $\eta=\sqrt{t}\xi$.
The last integral is finite since the integrand is bounded near $0$ (since $(1-e^{-x})^2 \leq 1-e^{-x} \leq x$ for $x>0$), and behaves like $|\eta|^{-2}$ at infinity.

\end{proof}

{\em Proof of  Theorem \ref{FCLT}:} {\em Step 1.} (Existence of c\`adl\`ag modification)
By stochastic Fubini Theorem,
\begin{equation}
\label{FR-decomp}
F_R(t)
=
\int_0^t\int_{\bR}
\varphi_{t,R}(s,y)u(s,y)L(ds,dy):=C_R(t)+M_R(t),
\end{equation}
where
\begin{align*}
C_R(t) &=
\int_0^t\int_{\bR}
\left(\varphi_{t,R}(s,y)-\mathbf 1_{[-R,R]}(y)\right) u(s,y)L(ds,dy), \\
M_R(t) &=
\int_0^t\int_{\bR}
\mathbf 1_{[-R,R]}(y)u(s,y)L(ds,dy).
\end{align*}

Let $0\leq s<t\leq T$. By It\^o's isometry and \eqref{mom-u},
\[
\begin{aligned}
\bE|C_R(t)-C_R(s)|^2 & \les 
\int_0^s\int_{\bR}
\left|
\varphi_{t,R}(r,y)-\varphi_{s,R}(r,y)
\right|^2  dydr  \\
&+ \int_s^t\int_{\bR}
\left|
\varphi_{t,R}(r,y)-\mathbf 1_{[-R,R]}(y)
\right|^2 dydr .
\end{aligned}
\]
By Lemma \ref{fclt-lemma1}  the first term is bounded by
$|t-s|^{3/2}$. By Lemma \ref{fclt-lemma2},
\[
\begin{aligned}
\int_s^t\int_{\bR}
\left|
\varphi_{t,R}(r,y)-\mathbf 1_{[-R,R]}(y)
\right|^2dydr
&=
\int_0^{t-s}
\|\varphi_{r,R}-\mathbf 1_{[-R,R]}\|_{L^2(\bR)}^2dr  \\
&\lesssim
\int_0^{t-s}r^{1/2}dr
\lesssim |t-s|^{3/2}.
\end{aligned}
\]
Therefore,
\begin{equation}
 \label{C_R}
\mathbb E|C_R(t)-C_R(s)|^2
\lesssim |t-s|^{3/2}.   
\end{equation}
By Kolmogorov-Chentsov criterion (Theorem 4.23 of \cite{kallenberg21}), $C_R$ has a
continuous modification. On the other hand, the process $M_R$ is a square-integrable stochastically continuous martingale. By Doob's Regularity Theorem (Theorem 3.40 of \cite{PZ07}), $M_R$ has a c\`adl\`ag modification. Consequently, $F_R=C_R+M_R$ has a c\`adl\`ag modification.

\medskip

{\em Step 2. (finite dimensional convergence)} In this step, we show that for any integer $m\geq 1$ and for any $t_1,\ldots,t_m>0$,
\[
\Big(\frac{1}{R^{1/2}} F_R(t_1),\ldots, \frac{1}{R^{1/2}}F_R(t_m)  \Big)
\stackrel{d}{\to} \Big(\cG(t_1),\ldots, \cG(t_m) \Big) \quad \mbox{as $R \to \infty$}.
\]
By Cram\'er-Wold theorem, this is equivalent to showing that for any $b_1,\ldots,b_m \in \bR$,
\[
X_R:=\frac{1}{R^{1/2}}\sum_{j=1}^{m}b_jF_R(t_j) \stackrel{d}{\to}X:=\sum_{j=1}^{m}b_j \cG(t_j) 
\quad \mbox{as $R \to \infty$}.
\]

As on page 4214 of \cite{BZ24}, it is enough to prove that $X_R/\tau_R \stackrel{d}{\to} Z$ as $R \to \infty$, where $\tau_R^2={\rm Var}(X_R) \to \tau^2={\rm Var}(X)$ and $Z \sim N(0,1)$. For this, we apply Proposition \ref{tara}:
\begin{equation}
d_{\mathrm{W}}\left(\frac{X_R}{\tau_R}, \mathcal{N}(0,1)\right)
\leq \gamma_1 + \gamma_2 + \gamma_3,
\end{equation}
where $\gamma_1$, $\gamma_2$, and $\gamma_3$ are defined as in \eqref{tara} with $F = X_R$.
Using the same method as in the proof of Theorem \ref{QCLT}, it can be proved that
$\gamma_i^p \lesssim R^{-(p-1)}$ for $i=1,2,3$.

\medskip
{\em Step 3. (Tightness).} From {\em Step 1}, for any $R>1$ and $ 0\leq s<t\leq T$, we have 
\[
\mathbb E\left|
\frac{C_R(t)-C_R(s)}{\sqrt R}
\right|^2
\lesssim |t-s|^{3/2}.
\]

By Theorem 23.7 of \cite{kallenberg21}, $\{R^{-1/2}C_R\}_{R\geq1}$ is tight in $C[0,T]$, as $C_R(0) = 0$ for any $R > 1$. Hence, $\{R^{-1/2}C_R\}_{R\geq1}$  is tight in $(D[0,T],U)$ where $U$ is the uniform topology. Since $U$ is stronger than $J_1$, $\{R^{-1/2}C_R\}_{R\geq1}$ is also tight in $(D[0,T],J_1)$. 

Next, we will prove that $\{R^{-1/2}M_R\}_{R\geq1}$ is tight in $(D[0,T],J_1)$. By Theorem 13.5 of \cite{billingsley99}, it is enough to prove that there exist $\beta>0$, $\alpha>\frac{1}{2}$ such that for all $0\leq r\le s\le t\leq T$,
\begin{equation}
\label{tight-related}
    \mathbb E\left[
\left|
\frac{M_R(s)-M_R(r)}{\sqrt R}
\right|^{2\beta}
\left|
\frac{M_R(t)-M_R(s)}{\sqrt R}
\right|^{2\beta}
\right]
\lesssim |t-r|^{2\alpha}.
\end{equation}

Recall that by hypothesis, $p\in(1,\frac32)$ is such that $m_p<\infty$ and $m_{2p}<\infty$. We set
\(\beta=\frac p2\) and \(\alpha=\frac{p+1}{4}>\frac{1}{2}\).
By Proposition \ref{ros1}, 
for any
\(0\le s\le t\le T\),
\begin{align}
\label{M_R-1}
 \bE|M_R(t)-M_R(s)|^{2p}
&\lesssim
\bE\left(
\int_s^t\int_{-R}^{R}|u(\theta,y)|^2dyd\theta
\right)^p
+
\bE\int_s^t\int_{-R}^{R}|u(\theta,y)|^{2p}dyd\theta .
\end{align}
Since \(p>1\), by H\"older's inequality, 
\begin{align}
\label{M_R-2}
\bE\left(
\int_s^t\int_{-R}^{R}|u(\theta,y)|^2dyd\theta
\right)^p
&\lesssim (t-s)^{p-1}(2R)^{p-1}
\bE\int_s^t\int_{-R}^{R}|u(\theta,y)|^{2p}dyd\theta .
\end{align}
{For the second term, by Fubini's theorem and}
\eqref{u-pth-moment}, we have
\begin{equation}
\label{M_R-3}
 \bE
\int_s^t\int_{-R}^{R}|u(\theta,y)|^{2p}dyd\theta
\lesssim R(t-s).
\end{equation}
{Combining} \eqref{M_R-1}, \eqref{M_R-2}, and
\eqref{M_R-3}, we obtain
\begin{equation}
    \label{M_R}
\bE|M_R(t)-M_R(s)|^{2p}
\lesssim R^p(t-s)^p+R(t-s).
\end{equation}

{Now let \(0\le r\le s\le t\le T\).}
Using Jensen's inequality for conditional expectation for the concave function $\varphi(x)=x^{p/2},x>0$, and the fact that $M_R$ is a martingale whose predictable quadratic variation is $\langle M_R \rangle(t) =m_2 \int_0^t \int_{-R}^R  |u(\theta,y)|^2 dy d\theta$, we have:
\begin{align}
\nonumber
\bE \left[  \left| M_R(t)-M_R(s) \right|^p \bigg| \cF_s  \right] 
&\leq \left( \bE \left[  \big(M_R(t)-M_R(s) \big)^2 \bigg| \cF_s  \right] \right)^{p/2}\\
\nonumber
& = \left(\bE \bigg[M_R^2(t)-M_R^2(s)|\cF_s\bigg] \right)^{p/2}=\left(\bE \bigg[\langle M_R\rangle(t)-\langle M_R\rangle(s)|\cF_s\bigg] \right)^{p/2}\\
\label{cond-bound}
&=
\left( m_2
\bE\left[
\int_s^t\int_{-R}^{R}|u(\theta,y)|^2dyd\theta
\,\bigg|\,\cF_s
\right]
\right)^{p/2}.
\end{align}

Conditioning on \(\cF_s\), and then using the Cauchy-Schwarz inequality, we have:
\begin{align*}
 &\qquad \bE\left[
|M_R(s)-M_R(r)|^p |M_R(t)-M_R(s)|^p
\right]  \\
&=
\bE\left[
|M_R(s)-M_R(r)|^p
\bE\left[ \left|M_R(t)-M_R(s)
\right|^p
\bigg|\cF_s
\right]
\right]\\
&\le \left(\bE|M_R(s)-M_R(r)|^{2p}\right)^{1/2}
\left\{
\bE \left[ \left(
\bE\left[ \left|M_R(t)-M_R(s)
\right|^p
\bigg|\cF_s
\right]
\right)^2 \right]
\right\}^{1/2}.
\end{align*}

By \eqref{cond-bound}, and Jensen's inequality for conditional expectation,
\begin{align*}
\bE\left[ \left(
\bE\left[ \left|M_R(t)-M_R(s)
\right|^p
\bigg|\cF_s
\right]
\right)^2 \right]
&\lesssim
\bE \left[ \left(
\bE\left[
\int_s^t\int_{-R}^{R}|u(\theta,y)|^2dyd\theta
\,\bigg|\,\cF_s
\right]
\right)^p \right] \\
& \les \bE\left[ \left(
\int_s^t\int_{-R}^{R}|u(\theta,y)|^2dyd\theta
\right)^p
\right]\\
&\lesssim R^p(t-s)^p,
\end{align*}
where the last inequality {follows from} \eqref{M_R-2} and \eqref{M_R-3}.
Therefore,
\begin{align*}
    &\qquad \bE\left[
|M_R(s)-M_R(r)|^p |M_R(t)-M_R(s)|^p
\right] \\
&\lesssim
\left(R^{p/2}(s-r)^{p/2}+R^{1/2}(s-r)^{1/2}\right)
R^{p/2}(t-s)^{p/2}.
\end{align*}
Dividing by \(R^p\), and using \(R\ge1\) and \(p>1\), we obtain
\[
\begin{aligned}
&\bE\left[
\left|
\frac{M_R(s)-M_R(r)}{\sqrt R}
\right|^p
\left|
\frac{M_R(t)-M_R(s)}{\sqrt R}
\right|^p
\right]  \\
&\quad \lesssim
(s-r)^{p/2}(t-s)^{p/2}
+
(s-r)^{1/2}(t-s)^{p/2}
\lesssim |t-r|^{(p+1)/2}.
\end{aligned}
\]
Since \(p=2\beta\) and \((p+1)/2=2\alpha\), this proves
\eqref{tight-related}.

\medskip
By Theorem~4.1 of
\cite{W80}, the addition map $(x,y)\mapsto x+y$ on $(D[0,T],J_1)\times(D[0,T],J_1)$ is continuous at every pair of
functions with no common discontinuities. Hence its restriction to
$C[0,T]\times(D[0,T],J_1)$ is continuous. Since $\{R^{-1/2}C_R\}_{R\geq1}$ is tight in $C[0,T]$ and $\{R^{-1/2}M_R\}_{R\geq1}$ is
tight in $(D[0,T],J_1)$, the sum  $\{R^{-1/2}F_R\}_{R\geq1}$ is tight in $(D[0,T],J_1)$.

\medskip

{\em Step 4.} Theorem 13.1 of \cite{billingsley99}, combined with Steps 2 and 3, shows that 
\[
\frac{1}{\sqrt R}F_R(\cdot)
\xrightarrow{d}
\cG(\cdot) \quad \mbox{in $(D[0,T],J_1)$ as $R \to \infty$}.
\]
By Theorem 23.9.(iii) of \cite{kallenberg21}, this convergence holds also in $(D[0,T],U)$, since $\cG$ is continuous. This concludes the proof of Theorem \ref{FCLT}.

\bigskip

{\em Acknowledgment.} The authors would like to thank Guangqu Zheng for reading the manuscript and providing useful comments.

\end{document}